\documentclass[11pt, reqno]{amsart}
\usepackage{amsmath, amsthm, amscd, amsfonts, amssymb, graphicx, color}
\usepackage{float}
\usepackage[bookmarksnumbered, colorlinks, plainpages]{hyperref}
\hypersetup{colorlinks=true,linkcolor=red, anchorcolor=green, citecolor=cyan, urlcolor=red, filecolor=magenta, pdftoolbar=true}

\newtheorem{theorem}{Theorem}[section]
\newtheorem{lemma}[theorem]{Lemma}
\newtheorem*{theorem*}{Question}
\newtheorem{proposition}[theorem]{Proposition}
\newtheorem{corollary}[theorem]{Corollary}
\theoremstyle{definition}
\newtheorem{definition}[theorem]{Definition}
\newtheorem{example}[theorem]{Example}

\newtheorem{problem}[theorem]{Problem}
\newtheorem{remark}[theorem]{Remark}
\numberwithin{equation}{section}

\begin{document}
	
	\setcounter{page}{1}

	\title[Convexity of the Berezin Range and Berezin Radius Inequalities ] {Convexity of Berezin Range and Berezin Radius Inequalities via a class of  Seminorm}
	\author[ P. Hiran Das, Athul  Augustine, Pintu Bhunia \MakeLowercase{and} P. Shankar]{ P. Hiran Das, Athul Augustine, Pintu Bhunia \MakeLowercase{and} P. Shankar}

	\address{P. Hiran Das, Department of Mathematics, Cochin University of Science And Technology,  Ernakulam, Kerala - 682022, India. }
	\email{\textcolor[rgb]{0.00,0.00,0.84}{hirandas939@gmail.com, hirandas073@cusat.ac.in}}
	
	\address{Athul Augustine, Department of Mathematics, Cochin University of Science And Technology,  Ernakulam, Kerala - 682022, India. }
	\email{\textcolor[rgb]{0.00,0.00,0.84}{athulaugus@gmail.com, athulaugus@cusat.ac.in}}

	\address{Pintu Bhunia, Department of Mathematics, SRM University AP, Amaravati 522240, Andhra Pradesh, India.}
	\email{\textcolor[rgb]{0.00,0.00,0.84}{pintubhunia5206@gmail.com}}
	
	\address{P. Shankar, Department of Mathematics, Cochin University of Science And Technology,  Ernakulam, Kerala - 682022, India.}
	
	\email{\textcolor[rgb]{0.00,0.00,0.84}{shankarsupy@gmail.com, shankarsupy@cusat.ac.in}}
	
	

\begin{abstract}
	
    Let $B(\mathcal{H})$ denote the 
    $C^*$-algebra of all bounded linear operators acting on a reproducing 
    kernel Hilbert space $\mathcal{H}(\Omega).$  In this paper, we introduce a new family of seminorms on $B(\mathcal{H})$, 
    called the $\sigma_t$-Berezin norm,
    defined as
    $$
    \|A\|_{{ber}_{\sigma_t}} 
    = \sup_{\lambda,\mu\in \Omega}
    \left\{
    \left(
    \left|\left\langle A\hat{k}_\lambda,\hat{k}_\mu\right\rangle\right|^p 
    \, \sigma_t \,
    \left|\left\langle A^*\hat{k}_\lambda,\hat{k}_\mu\right\rangle\right|^p
    \right)^{\frac{1}{p}}
    \right\},
    $$
    where $A\in B(\mathcal{H}), ~p \geq 1, ~t \in [0,1]$ and ~$\sigma_t$ denotes an interpolation path of a symmetric mean $\sigma$.
    We show that this family of seminorms characterizes invertible operators that are unitary. Several fundamental properties of the $\sigma_t$-Berezin norm are established, along with a collection of new inequalities that yield refined upper bounds for the Berezin radius of bounded linear operators, thereby improving  existing results in the literature.
    
     Furthermore, we investigate the convexity of the Berezin range  of operators  acting on weighted Hardy space and Fock space over $\mathbb{C}^n$. We characterised the convexity of the Berezin range of composition operator with elliptic automorphism and finite rank operators
     with different weights on the weighted Hardy space. We also characterized  convexity of the Berezin range of composition operator on 
	 Fock space over $\mathbb{C}^n$ with symbol $\phi(z)=Az$, where $A$ is a scalar matrix of order $n$.
\end{abstract}
	\maketitle
\section{Introduction}
    A \emph{reproducing kernel Hilbert space} (RKHS) $\mathcal{H}=\mathcal{H}(\Omega)$ on a nonempty set $\Omega$ is a Hilbert space of complex-valued functions such that, for each $\lambda  \in \Omega$, the point evaluation functional $E_{\lambda} : \mathcal{H} \to \mathbb{C}$ defined by
    $ E_{\lambda}(f) = f(\lambda)
    $ is bounded. By the Riesz representation theorem, for every  $\lambda \in \Omega$ there exists a unique element $k_\lambda \in \mathcal{H}$ such that $f(\lambda) = \langle f, k_\lambda \rangle,  \text{for all } f \in \mathcal{H}.$ The function $k_\lambda$ is called the \emph{reproducing kernel} of $\mathcal{H}$ at $\lambda$. The \emph{normalized reproducing kernel} at $\lambda$ is defined by
    $
    \hat{k}_\lambda = \frac{k_\lambda}{\|k_\lambda\|}.
    $ For more details on reproducing kernel Hilbert space, refer \cite{paulsen2016introduction}.

    For a bounded linear operator $A\in B(\mathcal{H})$, the Berezin transform (or Berezin symbol) of $A$ at $\lambda$ is defined by $\widetilde{A}(\lambda):= \langle A \hat{k}_\lambda, \hat{k}_\lambda \rangle$. The Berezin range (or Berezin set) and the Berezin radius (or Berezin number) of an operator $A$ are defined, respectively, as
    $$ \textit{Ber}(A) := \{\widetilde{A}(\lambda) : \lambda \in \Omega\}\quad \text{and}\quad \textit{ber}(A) := \sup_{\lambda \in \Omega}\{|\widetilde{A}(\lambda)|\}. $$
    The Berezin transform of an operator on a reproducing kernel Hilbert space was first introduced by F. A. Berezin \cite{berezin1972covariant}. 
    It follows directly from the definition that the Berezin range is contained in the numerical range of $A$, and hence the Berezin radius satisfies $\textit{ber}(A) \leq \|A\|$. By the Toeplitz–Hausdorff theorem \cite{gustafson1970toeplitz}, the numerical range of a bounded operator is always convex. 
    However, the Berezin range need not be convex in general. 
    The convexity of the Berezin range for operators on reproducing kernel Hilbert spaces has been studied in several works; see, for instance, \cite{augustine2025convexity,augustine2023composition,augus2024,cowen22,karaev2013reproducing,anirben}. For a detailed study of  the Berezin radius inequalities, we refer to \cite{newfam,bakherad2021some,bhunia2026strengthening,bhunia2023some,newnorm,bhunia2024berezin,garayev2023weighted,mainpaper,zamani2024berezin}.

    For a bounded linear operator $A\in B(\mathcal{H})$,
    the Berezin norm is defined by
    $$
    \|A\|_{ber}=\sup_{\lambda,\mu \in \Omega}|\langle A\hat{k}_{\lambda},\hat{k}_{\mu}\rangle|.
    $$
    The Berezin norm defines an operator norm on $B(\mathcal{H})$. Moreover, it is easy to observe that
    $$
    \textit{ber}(A) \leq \|A\|_{ber} \leq \|A\|.
    $$
    Recently, Bhunia et al.~\cite{equality} proved that for a positive operator $A\in B(\mathcal{H})$
    $$
    \textit{ber}(A) = \|A\|_{ber} .
    $$
    To introduce the class of seminorms considered in this paper, we next recall the concept of means and interpolation paths.

    A mean is a nonnegative function $\sigma:[0,\infty)\times [0,\infty)\rightarrow [0,\infty)$ that satisfies the following conditions (see \cite{bhatia,geometric}) :
\begin{enumerate}
    \item[(i)]$\sigma(a,b)\geq 0;$  
    \item[(ii)]if $a\leq b$, $a\leq \sigma(a,b)\leq b;$  
    \item[(iii)]$\sigma(a,b)$ is monotone increasing in both a and b;  
    \item[(iv)]For $\alpha>0$, $\sigma(\alpha a,\alpha b)=\alpha\sigma(a,b)$ (homogeneity);  
    \item[(v)] $\sigma(a,b)$ is continuous.  
\end{enumerate}

    For convenience, throughout this paper, we denote $\sigma(a, b)$ by $a \sigma b$.
    A mean, $\sigma$ is called a symmetric mean if $a\sigma b= b\sigma a$. For a symmetric mean $\sigma$, a parameterized operator mean $\sigma_t$, for each $t\in[0,1]$, is called an interpolation path for $\sigma$ if it satisfies the following conditions:
\begin{enumerate}
    \item[(i)]$a\sigma_0 b=b,a\sigma_1 b=a,a \sigma_\frac{1}{2} b=a\sigma b~ \forall ~ a,b\geq0;$  
    \item[(ii)]$(a\sigma_t b)\sigma(a\sigma_s b) = a\sigma_\frac{t+s}{2} b, \forall~t,s \in [0,1];$  
    \item[(iii)]For each $0\leq t \leq 1$, $\sigma_t$ is increasing in each of its components.  
\end{enumerate}
    The arithmetic, geometric, and harmonic means are  examples of symmetric means. 
    For two nonnegative numbers $a$ and $b$, the arithmetic mean is defined by 
    $
    a \nabla b = \frac{a+b}{2},
    $
    and its interpolation path is given by
    $
    a \nabla_t b = t a + (1-t)b, \quad 0 \le t \le 1.
    $
    For additional details, see \cite{mond}.

    In \cite{mainpaper},  Nayak and  Bhunia introduced a new norm, called the $t$-Berezin norm, on  $B(\mathcal{H})$.

\begin{definition}\label{definition1.1}\cite{mainpaper}
	Let $A\in B(\mathcal{H})$. Then for each $t\in [0,1],$ the t-Berezin norm is defined as 
	$$
	\|A\|_{ber_t}
	=
	\sup_{\lambda,\mu \in \Omega}
	\left\{
	t \bigl|\langle A \hat{k}_{\lambda}, \hat{k}_{\mu} \rangle\bigr|
	+
	(1-t) \bigl|\langle A^{*} \hat{k}_{\lambda}, \hat{k}_{\mu} \rangle\bigr|
	\right\},
	\quad t \in [0,1].
	$$
\end{definition}

    It can be verified that, for $t \in [0,1]$, the $t$-Berezin norm defines a norm on ${B}(\mathcal{H})$. We now extend this definition to general interpolation paths of symmetric means.

\begin{definition}\label{definition1.2}
	Let $A\in B(\mathcal{H})$ and for $t\in [0,1],$ let $\sigma_t$ be an interpolation path of the symmetric mean $\sigma$. Then the $\sigma_t$-Berezin norm of $A$, denoted by $\|A\|_{{ber}_{\sigma_t}},$ is defined as
	$$\|A\|_{{ber}_{\sigma_t}} = \sup_{\lambda,\mu\in \Omega}\left\lbrace\left(\left|\left\langle A\hat{k}_\lambda,\hat{k}_\mu\right\rangle\right|^p~\sigma_t ~ \left|\left\langle A^*\hat{k}_\lambda,\hat{k}_\mu\right\rangle\right|^p\right)^{\frac{1}{p}}\right\rbrace,$$
	where $p\geq 1$.
\end{definition}
    Clearly, Definition \ref{definition1.1} follows from Definition \ref{definition1.2} by taking $\sigma = \nabla $, the arithmetic mean and $p=1$.

    In this paper, we investigate the fundamental properties of the $\sigma_t$-Berezin norm and generalize several known bounds for the Berezin radius of bounded linear operators on $B(\mathcal{H})$ by deriving new inequalities involving the $\sigma_t$-Berezin norm. In addition, we apply this norm to characterize invertible operators that are unitary. We further study the convexity of the Berezin range for classes of operators acting on weighted Hardy space and Fock space over $\mathbb{C}^n$. In particular, we studied weighted Hardy space with weight $\beta_n^2 = (\frac{1}{\beta})^n$ for $\beta \in (0,1]$ and characterized the convexity of composition operator with symbol $\phi(z)=\eta z, $ where $\eta \in \overline{\mathbb{D}}.$ Also, we discussed the convexity of finite rank operators on the weighted Hardy space with the  same weight. Further, we characterized the convexity of the Berezin range of composition operators on the Fock space with symbols $\phi(z)=Az,$ where $A=\lambda I, \lambda \in \overline{\mathbb{D}}$ and $\phi(z)=A_kz$ such that, for a fixed $k\in\{1,2,...,n\},$ $A_k=[a_{kj}]$ is a diagonal matrix with entries $a_{kk}=a+ib$ and $a_{jj}=1$ $\forall$ $k\not=j$, where $a^2+b^2\leq 1$.   Our results extend and unify several existing results on the convexity of the Berezin range.
\section{preliminaries}
	Before proving the main results of this paper, we recall  some auxiliary results that will be used throughout the paper.

    Let $A \in B(\mathcal{H})$, then the polar decomposition of $A$ is given by 
    $
    A=\mathcal{U}|A|,
    $
  where $\mathcal{U}$ is a partial isometry and $|A|=(A^*A)^{\frac{1}{2}}.$        
\begin{lemma}\cite{generalisedinequality}
	Let $A, B \in B(\mathcal{H})$ have polar decompositions 
	$$
	A=\mathcal{U}|A|, \qquad  	B=\mathcal{V}|B|
	$$
	and let $f,g :[0,\infty) \to [0,\infty)$ be nonnegative continuous. Then for all $x, y \in \mathcal{H}$  the following inequality holds:
\begin{equation}\label{2.1}
	\left|\left\langle g\left(|B|\right)f\left(|A^*|\right)\mathcal{U}x,y\right\rangle\right|^2\leq \left\langle f^2(|A|)x,x\right\rangle \left\langle g^2(|B|)y,y\right\rangle,
\end{equation}
\begin{equation}\label{2.2}
	\left|\left \langle g\left(|A|\right)f\left(|B^*|\right)\mathcal{V}x,y\right\rangle\right|^2 \leq \left\langle f^2(|B|)x,x\right\rangle \left\langle g^2(|A|)y,y\right\rangle.
\end{equation}
    In particular, if $f$ and $g$ are nonnegative continuous functions on $[0,\infty)$ such that $f(t)g(t)=t,$ then 
\begin{equation}\label{2.3}
	\left|\left\langle Ax,y\right\rangle\right|^2\leq \left\langle f^2(|A|)x,x\right\rangle \left\langle g^2(|A^*|)y,y\right\rangle.
\end{equation}
\end{lemma}
\begin{lemma}\cite{complex}
	Let $a_i$ be positive numbers for $i=1,2,...,n.$ Then for $p \geq 1$
\begin{equation}\label{2.4}
	\left(\sum_{i=1}^{n}a_i\right)^p \leq n^{p-1}\sum_{i=1}^{n} {a_i}^p.
\end{equation}
\end{lemma}
\begin{lemma}\cite[Theorem 5]{notes}
	Let $A,B \in B(\mathcal{H})$ be such that $|A|B=B^*|A|$, and let $f$ and $g$ be two nonnegative continuous functions defined on $[0,\infty)$ such that $f(t)g(t)=t$ for every $t\geq 0$. Then
\begin{equation}\label{2.5}
	|\langle AB x,y\rangle|^2 \leq r(B)\langle f^2(|A|)x,x\rangle \langle g^2(|A^*|)y,y\rangle,
\end{equation}
    for every $x,y\in \mathcal{H}$, where $r(B)$ denotes the spectral radius of $B$.
\end{lemma}
\begin{lemma}\cite{simon}
	If $T\in B(\mathcal{H})$ is positive, then 
\begin{equation}\label{2.6}
	\langle Tx,x\rangle^p \leq \langle T^p x,x\rangle,~ \forall ~ p\geq 1,~ \forall ~ x\in \mathcal{H},~ \|x\|=1.	
\end{equation}
	The inequality \eqref{2.6} is reversed when $0\leq p\leq 1$.
\end{lemma}
\begin{lemma}\cite{buzano}
	Let $x,y,e \in \mathcal{H}$ with $\|e\|=1$, then
\begin{equation}\label{2.7}
	|\left\langle x,e\right\rangle\left\langle e,y\right\rangle| \leq \frac{1}{2}\left(\left|\left\langle x,y \right\rangle \right|+\|x\|\|y\|\right).
\end{equation}
\end{lemma}
    A well known bound for the Berezin radius is given in the following theorem.
\begin{theorem}\cite[Corollary        3.4]{upperbound}\label{sigma}
	Let $A_i,B_i,X_i\in B(\mathcal{H})$ $(i=1,2,...,n)$ and $0<\alpha<1$. Then for $p\geq 1$,
	$$\textit{ber}^p\left(\sum_{i=1}^n A_i^*X_iB_i\right)\leq \frac{n^{p-1}}{2}\textit{ber}\left(\sum_{i=1}^n \left(\left[A_i^*|X_i^*|^{2(1-\alpha)}A_i\right]^p+\left[B_i^*|X_i|^{2\alpha}B_i\right]^p\right)\right).$$
\end{theorem}
	The following are some special cases of Theorem \ref{sigma}.
\begin{corollary}\cite[Corollary 3.5]{upperbound}\label{bound}
	Let $A,B,X\in B(\mathcal{H})$. Then the following results hold:
\begin{itemize}
	\item[(i)]$\textit{ber}^p(A)\leq \frac{1}{2}\textit{ber}\left(|A|^p+|A^*|^p\right)\quad \forall~p\geq 1,$
	\item[(ii)]$\textit{ber}(A^*XB)\leq \frac{1}{2}\textit{ber}\left(|A|^2+|B|^2\right),$
	\item[(iii)]$\textit{ber}(A^*B)\leq \frac{1}{2}\textit{ber}\left(A^*|X^*|A+B^*|X^*|B\right).$
\end{itemize}
\end{corollary}
	
	\section{The $\sigma_t$- Berezin norm of Operators}
    In the beginning, we discuss some basic properties of the $\sigma_t$-Berezin norm, which  immediately follow from  Definition \ref{definition1.2}.

\begin{proposition}
	Let $A\in B(\mathcal{H})$. Then the following results hold:
\begin{itemize}
	\item[$(i)$]$\|A\|_{{ber}_{\sigma_t}} =\|A^*\|_{{ber}_{\sigma_t}} $.
	\item[$(ii)$]$\|A\|_{{ber}_{\sigma_t}} =0$ if and only if $A=0$.			
	\item[$(iii)$]$\|\lambda A\|_{{ber}_{\sigma_t}} =|\lambda|\|A\|_{{ber}_{\sigma_t}} $ for all $\lambda\in\mathbb{C}$.
	\item[$(iv)$] $\|A\|_{{ber}_{\sigma_t}}=\|A\|_{{ber}_{\sigma_{1-t}}}$.	
\end{itemize}
\end{proposition}
From this proposition, we can say that     $\|A\|_{{ber}_{\sigma_t}}$ defines a seminorm on $B(\mathcal{H})$ as it satisfies all norm axioms except triangle inequality. Also it is clear that it satisfies the inequality
\begin{equation}\label{inequality}
  	\textit{ber}(A) \leq \|A\|_{{ber}_{\sigma_t}} \leq \|A\|_{ber}.
\end{equation}
	   
Let $A \in {B({H})}$, we define  $\widetilde{c}(A):=\inf_{\lambda ,\mu \in \Omega}|\langle A\hat{k}_\lambda,\hat{k}_\mu\rangle |$. In the following theorem we obtain a lower bound for the $\sigma_t $-Berezin norm in terms of $\widetilde{c}(\cdot)$.
\begin{theorem}
    Let $A \in B\mathcal{(H)}$. Then for $p \geq 1$
    $$
    \max\left \{\widetilde{c}^p(A) ~ \sigma_{t} ~ \|A\|^p_{ber},\|A\|^p_{ber} ~ \sigma_{t} ~ \widetilde{c}^p(A)\right\}\leq \|A\|^p_{ber_{\sigma_{t}}}.
    $$
\end{theorem}
\begin{proof}
    We have 
\begin{equation*}
 \begin{split}
    \|A\|^p_{ber_{\sigma_t}}&=\sup_{\lambda,\mu \in\Omega}\left\{|\langle A\hat{k}_\lambda,\hat{k}_\mu\rangle |^p ~ \sigma_t ~ |\langle A^*\hat{k}_\lambda,\hat{k}_\mu\rangle |^p\right\}\\
    &\geq \widetilde{c}^p(A) ~ \sigma_{t} ~ |\langle A^*\hat{k}_\lambda,\hat{k}_\mu\rangle |^p.
\end{split}		
\end{equation*}
    Taking the supremum over all $\lambda,\mu \in \Omega, $ we get
\begin{equation}\label{3.4}
   	\|A\|^p_{ber_{\sigma_t}}\geq  \widetilde{c}^p(A) ~ \sigma_{t} ~ \|A\|_{ber}.
\end{equation}
    In a similar way, we can have
  	\begin{equation}\label{3.3}
    \|A\|^p_{ber_{\sigma_t}}\geq  \|A\|_{ber} ~ \sigma_{t} ~ \widetilde{c}^p(A).
\end{equation}
    By combining $(\ref{3.4})$ and $(\ref{3.3}),$ we obtain the desired inequality.
\end{proof}
Next result gives a complete characterization of the equality case for the Berezin norm.
\begin{theorem}
	Let $A\in B(\mathcal{H})$. Then the following two statements are equivalent:
\begin{itemize}
	\item[$(i)$]$\|A\|_{{ber}_{\sigma_t}} =\|A\|_{ber}$.
	\item[$(ii)$] There exist sequences $\{\lambda_n\}$ and $\{\mu_n\}$ in $\Omega$ such that 
	$$\lim_{n\rightarrow \infty}|\langle A\hat{k}_{\lambda_{n}},\hat{k}_{\mu	_{n}}\rangle|=\lim_{n\rightarrow \infty}|\langle A^*\hat{k}_{\lambda_{n}},\hat{k}_{\mu	_{n}}\rangle |=\|A\|_{ber}. $$ 
\end{itemize} 
\end{theorem}
\begin{proof}
	First, we prove $ (i) \implies (ii).$ Suppose $\|A\|_{{ber}_{\sigma_t}} =\|A\|_{ber}$ holds. Since $A$ is bounded, there exist sequences of normalised kernels $\{\hat{k}_{\lambda_{n}}\}$ and $\{\hat{k}_{\mu_{n}}\}$ in $\mathcal{H}$ such that 
	$$ \lim_{n\rightarrow \infty}(|\langle A\hat{k}_{\lambda_{n}},\hat{k}_{\mu	_{n}}\rangle |^{p}\sigma_t|\langle A^*\hat{k}_{\lambda_{n}},\hat{k}_{\mu	_{n}}\rangle |^{p})=\|A\|^{p}_{{ber}_{\sigma_t}}.$$
	Therefore, we have 
\begin{equation*}
\begin{split}
	\|A\|^{p}_{ber} &=\|A\|^{p}_{{ber}_{\sigma_t}}=\lim_{n\rightarrow \infty}(|\langle A\hat{k}_{\lambda_{n}},\hat{k}_{\mu	_{n}}\rangle |^{p}\sigma_t|\langle A^*\hat{k}_{\lambda_{n}},\hat{k}_{\mu_{n}}\rangle |^{p}).
\end{split}
\end{equation*}
     For $t=0$, we have $\|A\|_{ber}=\lim_{n\rightarrow \infty}(|\langle A^*\hat{k}_{\lambda_{n}},\hat{k}_{\mu	_{n}}\rangle | $ and for $t=1$, we have $\|A\|_{ber}=\lim_{n\rightarrow \infty}(|\langle A\hat{k}_{\lambda_{n}},\hat{k}_{\mu	_{n}}\rangle|.$
 	Therefore,  $\lim_{n\rightarrow \infty}|\langle A\hat{k}_{\lambda_{n}},\hat{k}_{\mu	_{n}}\rangle|=\lim_{n\rightarrow \infty}|\langle A^*\hat{k}_{\lambda_{n}},\hat{k}_{\mu_{n}}\rangle |=\|A\|_{ber}$.
 	
	Now to prove $(ii) \implies (i)$, assume $(ii)$ holds. We have 
\begin{equation*}
\begin{split}
	\|A\|^{p}_{{ber}_{\sigma_t}}&=\sup_{\lambda,\mu\in \Omega}\left\lbrace|\langle A\hat{k}_\lambda,\hat{k}_\mu\rangle|^p~\sigma_t ~ |\langle A^*\hat{k}_\lambda,\hat{k}_\mu\rangle|^p\right\rbrace \\
	&\geq \lim_{n\rightarrow \infty}|\langle A\hat{k}_{\lambda_{n}},\hat{k}_{\mu	_{n}}\rangle|^{p}\sigma_t\lim_{n\rightarrow \infty}|\langle A^*\hat{k}_{\lambda_{n}},\hat{k}_{\mu_{n}}\rangle |^{p}\\
	&=\|A\|^{p}_{ber}.
\end{split}
\end{equation*}
This completes the proof.
\end{proof}
	
    The following results establish an upper bound for the $\sigma_t$-Berezin norm.
\begin{theorem}
	Let $A\in B(\mathcal{H})$. If $\sigma_t\leq \nabla_t$, then for $p\geq 2$
\begin{itemize}
	\item[(i)]$
	\|A\|_{{ber}_{\sigma_t}}^{p}
	\le
	\textit{ber}\!\left(
	t\,|A^{*}|^{p} + (1-t)\,|A|^{p}
	\right).
	$
	\item[(ii)]$
	\textit{ber}^p(A)\leq \min_{t \in[0,1]}\textit{ber}\left(t|A|^p +(1-t)|A^*|^p\right) .$
\end{itemize}
\end{theorem}
\begin{proof}
	We have
\begin{equation*}
\begin{split}
	\|A\|^p_{ber_{\sigma_t}}
	& =  \sup_{\lambda,\mu\in\Omega}\left\lbrace |\langle A\hat{k}_\lambda,\hat{k}_\mu\rangle |^p~ \sigma_{t}~ |\langle A^*\hat{k}_\lambda,\hat{k}_\mu\rangle |^p\right\rbrace\\
	& \leq \sup_{\lambda\in\Omega}\left\lbrace \|A\hat{k}_\lambda\|^p~ \sigma_{t}~ \|A^*\hat{k}_\lambda\|^p\right\rbrace\\
	&=\sup_{\lambda\in\Omega}\left\lbrace \left\langle|A|^2\hat{k}_\lambda,\hat{k}_\lambda\right\rangle^\frac{p}{2}~ \sigma_{t}~\left\langle|A^*|^2\hat{k}_\lambda,\hat{k}_\lambda\right\rangle^\frac{p}{2}\right\rbrace\\ 
	&\leq \sup_{\lambda\in\Omega}\left\lbrace \left\langle|A|^p\hat{k}_\lambda,\hat{k}_\lambda\right\rangle~ \sigma_{t}~\left\langle|A^*|^p\hat{k}_\lambda,\hat{k}_\lambda\right\rangle\right\rbrace \qquad ( by \eqref{2.6})\\
	&\leq \sup_{\lambda\in\Omega}\left\lbrace t\left\langle|A|^p\hat{k}_\lambda,\hat{k}_\lambda\right\rangle +(1-t)\left\langle|A^*|^p\hat{k}_\lambda,\hat{k}_\lambda\right\rangle\right\rbrace\\
	&= \sup_{\lambda\in\Omega}\left\lbrace \left\langle \left(t |A|^p +(1-t)|A^*|^p
\right)\hat{k}_\lambda,\hat{k}_\lambda\right\rangle\right\rbrace\\  
	&=\textit{ber}\left(t|A|^p +(1-t)|A^*|^p\right),    
\end{split}
\end{equation*}
	Hence, the first inequality follows. The second inequality follows directly from \eqref{inequality}.
\end{proof}
 
\begin{remark}
	For $t=\frac{1}{2}$, we obtain
	$$
	\textit{ber}^{p}(A)
	\le
	\frac{1}{2}\textit{ber}\left(|A|^{p}+|A^*|^{p}\right),
	$$
	which is precisely the inequality stated in Corollary~\ref{bound}\,(i). Thus we have 
\begin{equation}\label{3.2}
	\textit{ber}^p(A)\leq \min_{t \in[0,1]}\textit{ber}\left(t|A|^p +(1-t)|A^*|^p\right) \leq \frac{1}{2}\textit{ber}\left(|A|^{p}+|A^*|^{p}\right).
\end{equation}
	To show the second inequality in (\ref{3.2}) is proper, we consider the finite-dimensional Hilbert space $\mathbb{C}^3$ and the operator
	$$
	A=\begin{pmatrix}
	0 & 2 & 2\\
	0 & 0 & 3\\
	0 & 0 & 0
	\end{pmatrix}.
	$$
	Then for $p=4, \min_{t \in[0,1]}\textit{ber}(t|A|^4+(1-t)|A^*|^4)=\frac{481}{6} \text{ and } \frac{1}{2}\textit{ber}(|A|^4+|A^*|^4)=\frac{185}{2}$. Therefore, for the matrix $A$
	$$
	\min_{t \in[0,1]}\textit{ber}\left(t|A|^p +(1-t)|A^*|^p\right) < \frac{1}{2}\textit{ber}\left(|A|^{p}+|A^*|^{p}\right).
	$$
	Hence, it is a stronger upper bound for the Berezin radius than the bound in  Corollary~\ref{bound}\,(i) . 
\end{remark}
\begin{theorem}
	If $A,B \in B(\mathcal{H}), $ with $|A|B=B^{*}|A|$, and let $f$ and $g$ be two nonnegative continuous functions defined on $[0,\infty )$ such that $f(t)g(t)=t$ for every $t \geq 0$. If $\sigma_t \leq \nabla_t $, then for   $p \geq 1$ 
	$$
	\|AB\|^p_{ber_{\sigma_t}} \leq  r^{\frac{p}{2}}(B)
	\textit{ber}^{\frac{1}{2}}\left( tf^2(|A|)^p+(1-t)|g^2(|A^*|)^p\right)
	\textit{ber}^{\frac{1}{2}}\left(tg^2(|A^*|)^p+(1-t)f^2(|A|)^p\right).
	$$
\end{theorem}
\begin{proof}
	Consider
\begin{equation*}
\begin{split}
	&|\langle AB\hat{k}_\lambda,\hat{k}_\mu\rangle|^p~\sigma_t~|\langle (AB)^*\hat{k}_\lambda,\hat{k}_\mu\rangle|^p\\
	&\leq \left( r^{\frac{1}{2}}(B)\langle f^2(|A|)\hat{k}_\lambda,\hat{k}_\lambda \rangle^{\frac{1}{2}} 
	\langle g^2(|A^*|)\hat{k}_\mu,\hat{k}_\mu \rangle^{\frac{1}{2}} \right)^{p} \\
	&\qquad\qquad\sigma_t  
	\left( r^{\frac{1}{2}}(B)\langle g^2(|A^*|)\hat{k}_\lambda,\hat{k}_\lambda \rangle^{\frac{1}{2}} 
	\langle f^2(|A|)\hat{k}_\mu,\hat{k}_\mu \rangle^{\frac{1}{2}} \right)^{p}  ( by \eqref{2.5} )\\		
	&= r^{\frac{p}{2}}(B)\left( \langle f^2(|A|)\hat{k}_\lambda,\hat{k}_\lambda \rangle^{\frac{p}{2}} \langle g^2(|A^*|)\hat{k}_\mu,\hat{k}_\mu \rangle^{\frac{p}{2}} 
	~ \sigma_t ~ \langle g^2(|A^*|)\hat{k}_\lambda,\hat{k}_\lambda \rangle^{\frac{p}{2}} 
	\langle f^2(|A|)\hat{k}_\mu,\hat{k}_\mu \rangle^{\frac{p}{2}} \right) \\
	&\leq r^{\frac{p}{2}}(B)\left(t\langle f^2(|A|)\hat{k}_\lambda,\hat{k}_\lambda \rangle^{\frac{p}{2}} 
	\langle g^2(|A^*|)\hat{k}_\mu,\hat{k}_\mu \rangle^{\frac{p}{2}} 
	+(1-t)\langle g^2(|A^*|)\hat{k}_\lambda,\hat{k}_\lambda \rangle^{\frac{p}{2}} 
	\langle f^2(|A|)\hat{k}_\mu,\hat{k}_\mu \rangle^{\frac{p}{2}} 
	\right) \\
    &\leq r^{\frac{p}{2}}(B)
	\left\{
	t\langle f^2(|A|)\hat{k}_\lambda,\hat{k}_\lambda \rangle^{p}
	+(1-t)\langle g^2(|A^*|)\hat{k}_\lambda,\hat{k}_\lambda \rangle^{p}
	\right\}^{\frac{1}{2}}\\
	&\qquad\qquad\times\left\{
	t\langle g^2(|A^*|)\hat{k}_\mu,\hat{k}_\mu \rangle^{p}				+(1-t)\langle f^2(|A|)\hat{k}_\mu,\hat{k}_\mu \rangle^{p}
	\right\}^{\frac{1}{2}} \\
	&\qquad\text{(by the Cauchy--Schwarz inequality)} \\
	&= r^{\frac{p}{2}}(B)
	\left\langle 
	(t f^{2p}(|A|)+(1-t)\psi^{2p}(|A^*|))\hat{k}_{\lambda},\hat{k}_{\lambda}
	\right\rangle^{\frac{1}{2}}
	\left\langle 
	(t g^{2p}(|A^*|)+(1-t)f^{2p}(|A|))\hat{k}_{\mu},\hat{k}_{\mu}
	\right\rangle^{\frac{1}{2}} \\
	&\leq r^{\frac{p}{2}}(B)
	\textit{ber}^{\frac{1}{2}}\left(t f^{2p}(|A|)+(1-t)|g^{2p}(|A^*|)\right)
	\textit{ber}^{\frac{1}{2}}\left(tg^{2p}(|A^*|)+(1-t)f^{2p}(|A|)\right).
\end{split}
\end{equation*}
	Taking the supremum over all $\lambda,\mu \in \Omega$, we obtain the desired  inequality.
\end{proof}
\begin{remark}
	If we set $B=I$, $f(t)= g(t)=\sqrt{t}$ and $t=\tfrac12$ in the general inequality obtained above, then we obtain the bound for the Berezin radius stated in Corollary~\ref{bound}\,(i). Hence, Corollary~\ref{bound}\,(i) appears as a special case of the preceding result.
\end{remark}
\begin{theorem}
	Let $A_i, B_i \in {B(\mathcal{H})} $, $i=1,2,...,n $ with $A_i$ have polar decomposition 
	$$
	A_i=\mathcal{U}_{i}|A_i|,
	$$
	and let $f,g : [0,\infty) \to[0,\infty) 
	$ be nonnegative continuous functions. If $\sigma_t \leq\nabla_t$, then for $p \geq 1$
	\begin{equation*}
    \begin{split}
    \left\|\sum_{i=1}^{n}g\left(|B_i|\right)f\left(|A_i^*|\right)\mathcal{U}_{i}\right\|_{ber_{\sigma_t}} &\leq \! \frac{n^{p-1}}{2}\! \sum_{i=1}^{n}\Big( \!\textit{ber}\left( t f^{2p}(|A_i|)
    +\!(1-t)g^{2p}(|B_i|)\right)\\
    &\qquad \qquad +\textit{ber}\left( t g^{2p}(|B_i|)\!+\!(1-t)f^{2p}(|A_i|)\right)\Big).
    \end{split}
	\end{equation*}
In particular, if $f$ and $g$ are nonnegative continuous functions on $[0,\infty)$ such that $f(t)g(t)=t,$ then 
\begin{equation*}
\begin{split}
	\left\|\sum_{i=1}^{n}A_i\right\|^p_{ber_{\sigma_t}} &\leq \frac{n^{p-1}}{2} \sum_{i=1}^{n} \Big(\textit{ber}\left( t f^{2p}(|A_i|)+(1-t)g^{2p}(|A_i^*|)\right)\\
	&\qquad \qquad +\textit{ber}\left( t g^{2p}(|A_i^*|)+(1-t)f^{2p}(|A_i|)\right)\Big) . 
\end{split}
\end{equation*}
\end{theorem}
\begin{proof}
	We have 
\begin{equation*}
\begin{split}
	&\left|\left\langle\sum_{i=1}^{n}g\left(|B_i|\right)f\left(|A_i^*|\right)\mathcal{U}_i\hat{k}_{\lambda},\hat{k}_{\mu}\right\rangle\right|^p =\left|\sum_{i=1}^{n}\left\langle g\left(|B_i|\right)f\left(|A_i^*|\right)\mathcal{U}_i\hat{k}_{\lambda},\hat{k}_{\mu}\right\rangle\right|^p\\
	&\qquad \qquad \leq \left(\sum_{i=1}^{n}\left|\left\langle   g\left(|B_i|\right)f\left(|A_i^*|\right)\mathcal{U}_i\hat{k}_{\lambda},\hat{k}_{\mu}\right\rangle\right|\right)^p\\
	&\qquad \qquad \leq \left(\sum_{i=1}^{n}\left\langle f^2(|A_i|)\hat{k}_{\lambda},\hat{k}_{\lambda}\right\rangle^{\frac{1}{2}}\left\langle g^2(|B_i|)\hat{k}_{\mu},\hat{k}_{\mu}\right\rangle^{\frac{1}{2}}\right)^p \qquad ( by \eqref{2.1})\\
	&\qquad \qquad \leq n^{p-1}\sum_{i=1}^{n}\left\langle f^{2p}(|A_i|)\hat{k}_{\lambda},\hat{k}_{\lambda}\right\rangle^{\frac{1}{2}}\left\langle g^{2p}(|B_i|)\hat{k}_{\mu},\hat{k}_{\mu}\right\rangle^{\frac{1}{2}} \quad ( by  \eqref{2.4})\\
	&\qquad \qquad \leq \frac{n^{p-1}}{2}\sum_{i=1}^{n}\left\langle f^{2p}(|A_i|)\hat{k}_{\lambda},\hat{k}_{\lambda}\right\rangle+\left\langle g^{2p}(|B_i|)\hat{k}_{\mu},\hat{k}_{\mu}\right\rangle.\\
\end{split}
\end{equation*}
	Now consider
\begin{equation*}
\begin{split}
    &\left|\left\langle\sum_{i=1}^{n}g\left(|B_i|\right)f\left(|A_i^*|\right)\mathcal{U}_i\hat{k}_{\lambda},\hat{k}_{\mu}\right\rangle\right|^p ~\sigma_t~	\left|\left\langle\sum_{i=1}^{n}\left(g\left(|B_i|\right)f\left(|A_i^*|\right)\mathcal{U}_i\right)^*\hat{k}_{\lambda},\hat{k}_{\mu}\right\rangle\right|^p\\
	&\qquad \qquad  \leq\frac{n^{p-1}}{2}\Bigg(\bigg[\sum_{i=1}^{n}\left(\left\langle f^{2p}(|A_i|)\hat{k}_{\lambda},\hat{k}_{\lambda}\right\rangle+\left\langle g^{2p}(|B_i|)\hat{k}_{\mu},\hat{k}_{\mu}\right\rangle\right)\bigg] \\
	&\qquad \qquad \qquad \qquad \qquad  ~\sigma_t~ \bigg[\sum_{i=1}^{n}\left(\left\langle f^{2p}(|A_i|)\hat{k}_{\mu},\hat{k}_{\mu}\right\rangle+\left\langle g^{2p}(|B_i|)\hat{k}_{\lambda},\hat{k}_{\lambda}\right\rangle\right)\bigg]\Bigg)\\
	&\qquad \qquad  \leq\frac{n^{p-1}}{2}\!\Bigg(t\bigg[\sum_{i=1}^{n}\left(\left\langle f^{2p}(|A_i|)\hat{k}_{\lambda},\hat{k}_{\lambda}\right\rangle+\left\langle g^{2p}(|B_i|)\hat{k}_{\mu},\hat{k}_{\mu}\right\rangle\right)\bigg] \!\\
	&\qquad \qquad \qquad \qquad  +\!(1-t) \bigg[\sum_{i=1}^{n}\left(\left\langle f^{2p}(|A_i|)\hat{k}_{\mu},\hat{k}_{\mu}\right\rangle+\left\langle g^{2p}(|B_i|)\hat{k}_{\lambda},\hat{k}_{\lambda}\right\rangle\right)\bigg]\Bigg)\\
	&\qquad\qquad  = \frac{n^{p-1}}{2} \!\sum_{i=1}^{n}\bigg(\left\langle t f^{2p}(|A_i|) +(1-t)g^{2p}(|B_i|)\hat{k}_{\lambda},\hat{k}_{\lambda}\right\rangle\!\\
	&\qquad \qquad \qquad \qquad +\!\left\langle t g^{2p}(|B_i|)+(1-t)f^{2p}(|A_i|)\hat{k}_{\mu},\hat{k}_{\mu}\right\rangle \bigg)\\
	& \qquad \qquad  \leq \frac{n^{p-1}}{2}\sum_{i=1}^{n} \Big(\textit{ber}\left( t f^{2p}(|A_i|)+(1-t)g^{2p}(|B_i|)\right)\\
	&\qquad \qquad \qquad \qquad +\textit{ber}\left( t g^{2p}(|B_i|)+(1-t)f^{2p}(|A_i|)\right)\Big).
\end{split}
\end{equation*}
	Taking the supremum over all $\lambda,\mu \in \Omega,$ we get the desired result.\\
    The other inequality can be directly obtained by choosing
	$$f(t)g(t)=t \text{\quad and \quad} B_i=A_i^*.$$
\end{proof}
 \begin{remark}
	By setting $n=1,t=\frac{1}{2}$ and $f(t)=g(t)=\sqrt{t} $ in the above Theorem, we get the upper bound for the Berezin radius which is stated in Corollary~\ref{bound}\,(i). Hence, Corollary~\ref{bound}\,(i) appears as a special case of the preceding result.
\end{remark}

\begin{theorem}
	Let $A, B \in {B(\mathcal{H})} $  with $A$ have polar decomposition 
	$$
	A=\mathcal{U}|A|
	$$
	and let $f,g:[0,\infty) \to [0,\infty)$ be nonnegative continuous functions. Let $r,s > 1$ be conjugate exponents satisfying $\frac{1}{r}+\frac{1}{s}=1.$ If $\sigma_t \leq \nabla_t$, then for $p \geq 2$
\begin{equation*}
\begin{split}
	\left\|g\left(|B|\right)f\left(|A^*|\right)\mathcal{U}\right\|^p_{ber_{\sigma_t}} &\leq \textit{ber}\left(t\frac{1}{r} f^{pr}\left(|A|\right)+(1-t)\frac{1}{s}g^{ps}\left(|B|\right)\right)\\
	&\quad +\textit{ber}\left(t\frac{1}{s}g^{ps}\left(|B|\right)+(1-t)\frac{1}{r}f^{pr}\left(|A|\right)\right).
\end{split}
\end{equation*}
	In particular, if $f$ and $g$ are nonnegative continuous functions on $[0,\infty)$ such that $f(t)g(t)=t$, then 
\begin{equation*}
\begin{split}
	\left\|A\right\|^p_{ber_{\sigma_t}} &\leq \textit{ber}\left(t\frac{1}{r} f^{pr}\left(|A|\right)+(1-t)\frac{1}{s}g^{ps}\left(|A^*|\right)\right)\\
	&\quad+\textit{ber}\left(t\frac{1}{s}g^{ps}\left(|A^*|\right)+(1-t)\frac{1}{r}f^{pr}\left(|A|\right)\right).
\end{split}
\end{equation*}
\end{theorem}
\begin{proof}
	We have
\begin{equation*}
\begin{split}
    \left|\left\langle g\left(|B|\right)f\left(|A^*|\right)\mathcal{U} \hat{k} _{\lambda},\hat{k}_{\mu}\right\rangle\right|^p &\leq  \left\langle f^p\left(|A|\right)\hat{k}_{\lambda},\hat{k}_{\lambda}\right\rangle\left\langle g^p\left(|B|\right)\hat{k}_{\mu},\hat{k}_{\mu}\right\rangle\\
	&\leq \left\langle f^{pr}\left(|A|\right)\hat{k}_{\lambda},\hat{k}_{\lambda}\right\rangle^{\frac{1}{r}}\left\langle g^{ps}\left(|B|\right)\hat{k}_{\mu},\hat{k}_{\mu}\right\rangle^{\frac{1}{s}}.
    \end{split}
\end{equation*}
Applying Young's inequality with conjugate exponents $r,s >1~\left(\frac{1}{r}+\frac{1}{s}=1\right),$ gives
$$
	\left\langle f^{pr}\left(|A|\right)\hat{k}_{\lambda},\hat{k}_{\lambda}\right\rangle^{\frac{1}{r}}\left\langle g^{ps}\left(|B|\right)\hat{k}_{\mu},\hat{k}_{\mu}\right\rangle^{\frac{1}{s}}\leq \frac{1}{r}\left\langle f^{pr}\left(|A|\right)\hat{k}_{\lambda},\hat{k}_{\lambda}\right\rangle+\frac{1}{s}\left\langle g^{ps}\left(|B|\right)\hat{k}_{\mu},\hat{k}_{\mu}\right\rangle.
$$
	Now consider 
\begin{equation*}
\begin{split}
	\left|\left\langle g\left(|B|\right)f\left(|A^*|\right)\mathcal{U} \hat{k} _{\lambda},\hat{k}_{\mu}\right\rangle\right|^p&~\sigma_t~\left|\left\langle \left(g\left(|B|\right)f\left(|A^*|\right)\mathcal{U}\right)^* \hat{k} _{\lambda},\hat{k}_{\mu}\right\rangle\right|^p\\
	&  \leq \bigg(\frac{1}{r}\left\langle f^{pr}\left(|A|\right)\hat{k}_{\lambda},\hat{k}_{\lambda}\right\rangle+\frac{1}{s}\left\langle g^{ps}\left(|B|\right)\hat{k}_{\mu},\hat{k}_{\mu}\right\rangle\bigg) \\
	& \qquad \quad~\sigma_t~ \bigg(\frac{1}{r}\left\langle f^{pr}\left(|A|\right)\hat{k}_{\mu},\hat{k}_{\mu}\right\rangle+\frac{1}{s}\left\langle g^{ps}\left(|B|\right)\hat{k}_{\lambda},\hat{k}_{\lambda}\right\rangle\bigg) \\
	& \leq t\left( \frac{1}{r}\!\left\langle f^{pr}\left(|A|\right)\hat{k}_{\lambda},\hat{k}_{\lambda}\right\rangle\!+\!\frac{1}{s}\!\left\langle g^{ps}\left(|B|\right)\hat{k}_{\mu},\hat{k}_{\mu}\right\rangle \right) \!\\
	&\qquad +\!(1-t)\left( \frac{1}{r}\!\left\langle f^{pr}\left(|A|\right)\hat{k}_{\mu},\hat{k}_{\mu}\right\rangle\!+\!\frac{1}{s}\!\left\langle g^{ps}\left(|B|\right)\hat{k}_{\lambda},\hat{k}_{\lambda}\right\rangle\right)\\
	& =\left\langle \left(t\frac{1}{r} f^{pr}\left(|A|\right)+(1-t)\frac{1}{s}g^{ps}\left(|B|\right)\right)\hat{k}_{\lambda},\hat{k}_{\lambda} \right\rangle\\
	&\qquad +\left\langle\left(t\frac{1}{s}g^{ps}\left(|B|\right)+(1-t)\frac{1}{r}f^{pr}\left(|A|\right)\right)\hat{k}_{\mu},\hat{k}_{\mu}\right\rangle\\
	& \leq \textit{ber}\left(t\frac{1}{r} f^{pr}\left(|A|\right)+(1-t)\frac{1}{s}g^{ps}\left(|B|\right)\right)\\
	&\qquad +\textit{ber}\left(t\frac{1}{s}g^{ps}\left(|B|\right)+(1-t)\frac{1}{r}f^{pr}\left(|A|\right)\right).
\end{split}
\end{equation*}
    Taking the supremum over all $\lambda,\mu \in \Omega,$ we get the desired result.\\
    The other inequality can be directly obtained by choosing
	$$f(t)g(t)=t \text{\quad and \quad} B=A^* .$$ 
\end{proof}
\begin{remark}
	By setting $r=s=2, t=\frac{1}{2} \text{ and } f(t)=g(t)=\sqrt{t}$ for all $ t\geq 0,$ we get bound for  the Berezin radius stated in Corollary~\ref{bound}\,(i). Hence, Corollary~\ref{bound}\,(i) appears as a special case of the preceding result.
\end{remark}
\begin{theorem}\label{3.12}
	Let $A_i,B_i\in {B}\mathcal{(H)}$, $i=1,2,\dots,n$, and suppose that each $A_i$ has the polar decomposition
	$$
	A_i=\mathcal{U}_i|A_i|.
   	$$
   	Let $f,g:[0,\infty)\to[0,\infty)$ be nonnegative continuous functions.  
	If $\sigma_t\leq \nabla_t,$ then for $p\geq 1$
\begin{equation*}
\begin{split}
	\left\|\sum_{i=1}^{n} g(|B_i|)\,f(|A_i^*|)\,\mathcal{U}_i \right\|_{ber_{\sigma_t}}^{2p}
	&\leq\frac{n^{2p-1}}{2}\Bigg(\sum_{i=1}^{n}\textit{ber}\!\left(f^{2p}(|A_i|)g^{2p}(|B_i|)\right)\\
	&\qquad +\frac{1}{2}\sum_{i=1}^{n}\bigl(\|f^{4p}(|A_i|)\|+\|g^{4p}(|B_i|)\|\bigr)\Bigg).
\end{split}
\end{equation*}
	In particular, if $f$ and $g$ are nonnegative continuous functions on $[0,\infty)$ such that $f(t)g(t)=t$, then 
	\[
	\left\|\sum_{i=1}^{n} A_i \right\|_{ber_{\sigma_t}}^{2p}\leq\frac{n^{2p-1}}{2}\left(\sum_{i=1}^{n} \textit{ber}\!\left(f^{2p}(|A_i|)g^{2p}(|A_i^*|)\right)
	+\frac{1}{2}\sum_{i=1}^{n}\bigl(\|f^{4p}(|A_i|)\|+\|g^{4p}(|A_i^*|)\|\bigr)\right).
    \]
\end{theorem}
\begin{proof}
	For $\lambda,\mu\in\Omega$, we have
\begin{equation*}
\begin{split}
    &\left|\left\langle
	\sum_{i=1}^{n} g(|B_i|)f(|A_i^*|)\mathcal{U}_i \hat{k}_{\lambda},
	\hat{k}_{\mu}\right\rangle\right|^{2p}\\
	&\qquad\leq\left(\sum_{i=1}^{n}\left|\left\langle g(|B_i|)f(|A_i^*|)\mathcal{U}_i \hat{k}_{\lambda},\hat{k}_{\mu}\right\rangle\right|\right)^{2p} \\
	&\qquad\leq\left(\sum_{i=1}^{n}\langle f^{2}(|A_i|)\hat{k}_{\lambda},\hat{k}_{\lambda}\rangle^{1/2}
	\langle g^{2}(|B_i|)\hat{k}_{\mu},\hat{k}_{\mu}\rangle^{1/2}
	\right)^{2p} \\
    &\qquad\leq n^{2p-1}\sum_{i=1}^{n}
	\langle f^{2p}(|A_i|)\hat{k}_{\lambda},\hat{k}_{\lambda}\rangle
	\langle g^{2p}(|B_i|)\hat{k}_{\mu},\hat{k}_{\mu}\rangle \\
	&\qquad\leq \frac{n^{2p-1}}{2}\sum_{i=1}^{n}\left(\left|\left\langle
	f^{2p}(|A_i|)\hat{k}_{\lambda},g^{2p}(|B_i|)\hat{k}_{\mu}\right\rangle\right|+\|f^{2p}(|A_i|)\hat{k}_{\lambda}\|\|g^{2p}(|B_i|)\hat{k}_{\mu}\|\right)\quad (by (\ref{2.7})) \\
	&\qquad\leq\frac{n^{2p-1}}{2}\sum_{i=1}^{n}\left(\left|\left\langle f^{2p}(|A_i|)\hat{k}_{\lambda},g^{2p}(|B_i|)\hat{k}_{\mu}
	\right\rangle\right|+\frac{1}{2}\bigl(\|f^{4p}(|A_i|)\|+\|g^{4p}(|B_i|)\|\bigr)\right).
\end{split}
\end{equation*}
	Now consider
\begin{equation*}
\begin{aligned}
	&\left|\left\langle\sum_{i=1}^{n} g(|B_i|)f(|A_i^*|)\mathcal{U}_i \hat{k}_{\lambda},\hat{k}_{\mu}
	\right\rangle\right|^{2p}\,\sigma_t\,\left|\left\langle
	\sum_{i=1}^{n}\bigl(g(|B_i|)f(|A_i^*|)\mathcal{U}_i\bigr)^*
	\hat{k}_{\lambda},\hat{k}_{\mu}\right\rangle\right|^{2p} \\[6pt]
	&\qquad \leq\frac{n^{2p-1}}{2}\Bigg(\sum_{i=1}^{n}\left(\left|
	\left\langle f^{2p}(|A_i|)\hat{k}_{\lambda},g^{2p}(|B_i|)\hat{k}_{\mu}
	\right\rangle\right|+\frac{1}{2}\bigl(\|f^{4p}(|A_i|)\|+\|g^{4p}(|B_i|)\|\bigr)\right)\\
	&\qquad \qquad\sigma_t~ \sum_{i=1}^{n}\bigg(\left|\left\langle
	f^{2p}(|A_i|)\hat{k}_{\mu},g^{2p}(|B_i|)\hat{k}_{\lambda}
	\right\rangle\right|+\frac{1}{2}\sum_{i=1}^{n}
	\bigl(\|f^{4p}(|A_i|)\|+\|g^{4p}(|B_i|)\|\bigr)\bigg)\Bigg) \\
	&\qquad \leq\frac{n^{2p-1}}{2}\Bigg(t\Bigg[\sum_{i=1}^{n}
	\left|\left\langle f^{2p}(|A_i|)\hat{k}_{\lambda},
	g^{2p}(|B_i|)\hat{k}_{\mu}\right\rangle\right|+\frac{1}{2}
	\sum_{i=1}^{n}\bigl(\|f^{4p}(|A_i|)\|+\|g^{4p}(|B_i|)\|\bigr)
	\Bigg] \\
	&\qquad \qquad +(1-t)\Bigg[	\sum_{i=1}^{n}\left|\left\langle
	f^{2p}(|A_i|)\hat{k}_{\mu},g^{2p}(|B_i|)\hat{k}_{\lambda}
	\right\rangle\right|+\frac{1}{2}\sum_{i=1}^{n}\bigl(
	\|f^{4p}(|A_i|)\|+\|g^{4p}(|B_i|)\|\bigr)\Bigg]\Bigg) \\
	&\qquad \leq\frac{n^{2p-1}}{2}
	\Bigg(t\Bigg[\sum_{i=1}^{n}\left|\left\langle\hat{k}_{\lambda},
	f^{2p}(|A_i|)g^{2p}(|B_i|)\hat{k}_{\mu}\right\rangle
	\right|+\frac{1}{2}\sum_{i=1}^{n}\bigl(\|f^{4p}(|A_i|)\|+\|g^{4p}(|B_i|)\|\bigr)\Bigg] \\
	&\qquad \qquad+(1-t)\Bigg[\sum_{i=1}^{n}
    \left|\left\langle\hat{k}_{\mu},f^{2p}(|A_i|)g^{2p}(|B_i|)\hat{k}_{\lambda}\right\rangle\right|+\frac{1}{2}\sum_{i=1}^{n}\bigl(
	\|f^{4p}(|A_i|)\|+\|g^{4p}(|B_i|)\|\bigr)\Bigg]\Bigg) \\
	&\qquad \leq\frac{n^{2p-1}}{2}\left(\sum_{i=1}^{n}\|f^{2p}(|A_i|)g^{2p}(|B_i|)\|_{ber}+\frac{1}{2}\sum_{i=1}^{n}\bigl(\|f^{4p}(|A_i|)\|+\|g^{4p}(|B_i|)\|\bigr)\right) \\
	&\qquad =\frac{n^{2p-1}}{2}\left(\sum_{i=1}^{n}\textit{ber}\!\left(f^{2p}(|A_i|)g^{2p}(|B_i|)\right)+\frac{1}{2}\sum_{i=1}^{n}
    \bigl(\|f^{4p}(|A_i|)\|+\|g^{4p}(|B_i|)\|\bigr)\right).
\end{aligned}
\end{equation*}
	Taking the supremum over all $\lambda,\mu\in\Omega$, we get the desired result.\\
	The other inequality follows by choosing
    $$
	 f(t)g(t)=t
	 \quad \text{and} \quad
	 B_i=A_i^*.
	$$
\end{proof}
	By setting $n=1,t=\frac{1}{2}$ and $f(t)=g(t)=\sqrt{t} $ in Theorem~\ref{3.12}, we get the following upper bound for the Berezin radius.
\begin{corollary}
	Let $ A\in {B(\mathcal{H})} $. Then for $p\geq 1$
	$$
	\textit{ber}^{2p}\left(A\right) \leq \frac{1}{2}\left(\textit{ber}\left(|A|^p|A^*|^p\right)+\frac{1}{2}\left(\left\||A|\right\|^{2p}+\left\||A^*|\right\|^{2p}\right)\right).
	$$
\end{corollary}
	 The following result applies to any interpolation path of a symmetric mean $\sigma$.
\begin{theorem}
	Let $ A\in {B(\mathcal{H})} $. Then for $p\geq 1$
	$$
	\|A\|^p_{ber_{\sigma_t}} \leq   \textit{ber}^{\frac{p}{2}}\left(|A|^2+|A^*|^2\right).
	$$
\end{theorem}
\begin{proof}
	We have
\begin{equation*}
\begin{split}
    \|A\|^p_{ber_{\sigma_t}} & =  \sup_{\lambda,\mu \in\Omega}\left\lbrace |\langle A\hat{k}_\lambda,\hat{k}_\mu\rangle |^p~ \sigma_{t}~ |\langle A^*\hat{k}_\lambda,\hat{k}_\mu\rangle |^p\right\rbrace\\
	&\leq  \sup_{\lambda,\mu\in\Omega}\bigg\lbrace\left( |\langle \mathcal{R}(A)\hat{k}_\lambda,\hat{k}_\mu\rangle|+|\langle \mathcal{I}(A)\hat{k}_\lambda,\hat{k}_\mu\rangle|\right)^{p}\\
	&\qquad \qquad \qquad ~ \sigma_{t}~\left( |\langle \mathcal{R}(A^*)\hat{k}_\lambda,\hat{k}_\mu\rangle|+|\langle \mathcal{I}(A^*)\hat{k}_\lambda,\hat{k}_\mu\rangle|\right)^p\bigg\rbrace\\
	&\leq \sup_{\lambda,\mu\in\Omega}\left\lbrace\left( \| \mathcal{R}(A)\hat{k}_\lambda\|+\|\mathcal{I}(A)\hat{k}_\lambda\|\right)^p~ \sigma_{t}~\left( \| \mathcal{R}(A^*)\hat{k}_\lambda\|+\|\mathcal{I}(A^*)\hat{k}_\lambda\|\right)^{p}\right\rbrace\\
	&= \sup_{\lambda,\mu\in\Omega}\bigg\lbrace\left( \langle \mathcal{R}^2(A)\hat{k}_\lambda,\hat{k}_\lambda\rangle^{\frac{1}{2}}+\langle \mathcal{I}^2(A)\hat{k}_\lambda,\hat{k}_\lambda\rangle^{\frac{1}{2}}\right)^p\\
	&\qquad \qquad \qquad~ \sigma_{t}~\left( \langle \mathcal{R}^2(A^*)\hat{k}_\lambda,\hat{k}_\lambda\rangle^{\frac{1}{2}}+\langle \mathcal{I}^2(A^*)\hat{k}_\lambda,\hat{k}_\lambda\rangle^{\frac{1}{2}}\right)^p\bigg\rbrace\\
	&\leq  \sup_{\lambda,\mu\in\Omega}\left\lbrace\left(2 \langle \mathcal{R}^2(A)+ \mathcal{I}^2(A)\hat{k}_\lambda,\hat{k}_\lambda\rangle\right)^\frac{p}{2}~ \sigma_{t}~\left( 2 \langle \mathcal{R}^2(A^*)+ \mathcal{I}^2(A^*)\hat{k}_\lambda,\hat{k}_\lambda\rangle\right)^\frac{p}{2}\right\rbrace\\
	&= \sup_{\lambda,\mu\in\Omega}\left\lbrace\left\langle \left(|A|^2+ |A^*|^2\right)\hat{k}_\lambda,\hat{k}_\lambda\right\rangle^\frac{p}{2}~ \sigma_{t}~\left\langle\left(|A^*|^2+ |A|^2\right)\hat{k}_\lambda,\hat{k}_\lambda\right\rangle^\frac{p}{2}\right\rbrace\\
	&=\sup_{\lambda,\mu\in\Omega}\left\lbrace\left\langle \left(|A|^2+ |A^*|^2\right)\hat{k}_\lambda,\hat{k}_\lambda\right\rangle^\frac{p}{2}\right\rbrace\\
	&=\textit{ber}^{\frac{p}{2}}\left(|A|^2+|A^*|^2\right),\\
\end{split}
\end{equation*}
	as desired.
\end{proof}
	
Recently, Bhunia et al.~\cite[Theorem 5]{bhunia2023some} characterized unitary operators 
in terms of the Berezin number. More precisely, they showed that 
an invertible operator $A \in B(\mathcal{H})$ is unitary if and only if 
$
\textit{ber}(A^*A) \le 1 
 \text{ and }  
\textit{ber}\big((A^*A)^{-1}\big) \le 1.
$

This naturally raises the question of whether unitary operators can 
also be characterized using the newly defined $\sigma_t$-Berezin norm. 
The following result shows that this is indeed the case.

\begin{theorem}
	Suppose $A \in B\mathcal{(H)}$ is invertible. Then $A$ is unitary if and only if $\|A^*A\|_{ber_{\sigma_t}}\!\leq \!1$ and $\|(A^*A)^{-1}\|_{ber_{\sigma_t}} \!\leq\! 1.$
\end{theorem}
\begin{proof}
	The necessity is trivial, as it follows from the definition of $\sigma_t$-Berezin  norm. To show the sufficiency, we calculate $\|(A-A^{-1^*})\hat{k}_\lambda\|^2$. Using the inequalities $\textit{ber}(A^*A)\leq\|A^*A\|_{ber_{\sigma_t}}\!\leq \!1 $ and $\textit{ber}((A^*A)^{-1})\leq\|(A^*A)^{-1}\|_{ber_{\sigma_t}}\!\leq \!1 $, we obtain $\|(A-A^{-1^*})\hat{k}_\lambda\|=0$ for all $\lambda \in \Omega$. This gives $A^*=A^{-1}.$
\end{proof}

    \textbf{Inequalities for Operator Matrices:}
    
   Here we study the $\sigma_t$-Berezin norm of $2\times 2$ operator matrices
   $
   \begin{pmatrix}
   	A & B\\
   	C & D
   \end{pmatrix},
   $
   acting on the direct sum of reproducing kernel Hilbert spaces. We derive new
   bounds for the $\sigma_t$-Berezin norm of any arbitrary $ 2 \times 2 $ operator matrix in terms of the diagonal and
   off-diagonal entries. Furthermore, we characterize those diagonal invertible operator
   matrices that are also unitary.
   
	We begin with the following proposition.
\begin{proposition}
	Let $A,B \in {B\mathcal{(H)}}$. Then 
\begin{itemize}
	\item [$(i)$] If $\sigma_{t}\leq \nabla_t$, then  $\left\|\begin{pmatrix}
	A & 0 \\
	0 & B
	\end{pmatrix}\right\|_{ber_{\sigma_t}} \!\leq \max\left\{\|A\|_{{ber}_{\nabla_t}},\|B\|_{{ber}_{\nabla_t}}\right\}.$
	\item [$(ii)$]  $\left\|\begin{pmatrix}
	0 & A \\
	B & 0
	\end{pmatrix}\right\|_{ber_{\sigma_t}}=\left\|\begin{pmatrix}
	0 & B \\
	A & 0
	\end{pmatrix}\right\|_{ber_{\sigma_t}}.$
	\item[$(iii)$] $\left\|\begin{pmatrix}
	0 & A \\
	0 & 0
	\end{pmatrix}\right\|^p_{ber_{\nabla_t}}\leq \max \left\{t,1-t\right\}\|A\|^p_{ber}.$
\end{itemize}
\end{proposition}
\begin{proof}
	$(1)$ Let $T=\begin{pmatrix}
	A & 0 \\
	0 & B
	\end{pmatrix}.$ For any $(\lambda_1,\lambda_2),(\mu_1,\mu_2) \in \Omega  \times  \Omega$, let $\hat{k}_{(\lambda_1,\lambda_2)}=(k_{\lambda_1},k_{\lambda_2}),\\ \hat{k
	}_{(\mu_1,\mu_2)}=(k_{\mu_1},k_{\mu_2}) $ be two normalised reproducing kernels in $\mathcal{H}\oplus \mathcal{H}.$ Then 
\begin{equation*}
\begin{split}
	&\left(
	 |\langle T\hat{k}_{(\lambda_1,\lambda_2)}, \hat{k}_{(\mu_1,\mu_2)} \rangle|^p 
	~\sigma_t~ 
	|\langle T^*\hat{k}_{(\lambda_1,\lambda_2)}, \hat{k}_{(\mu_1,\mu_2)} \rangle|^p
	\right)^{\frac{1}{p}}\\
    &\quad = \left(|\langle A  k_{\lambda_1},k_{\mu_1}\rangle+\langle B k_{\lambda_2},k_{\mu_2}\rangle|^p ~\sigma_t ~ 	|\langle A^* k_{\lambda_1},k_{\mu_1}\rangle+\langle B^* k_{\lambda_2},k_{\mu_2}\rangle|^p\right)^{\frac{1}{p}}\\
	&\quad \leq \left(t|\langle A k_{\lambda_1},k_{\mu_1}\rangle+\langle B k_{\lambda_2},k_{\mu_2}\rangle|^p +(1-t)	|\langle A^* k_{\lambda_1},k_{\mu_1}\rangle+\langle B^* k_{\lambda_2},k_{\mu_2}\rangle|^p\right)^{\frac{1}{p}}\\
	&\quad \leq \left(t|\langle A k_{\lambda_1},k_{\mu_1}\rangle|^p+(1-t)	|\langle A^* k_{\lambda_1},k_{\mu_1}\rangle|^p\right)^{\frac{1}{p}}+\left(t|\langle B k_{\lambda_2},k_{\mu_2}\rangle|^p +(1-t)\langle B^* k_{\lambda_2},k_{\mu_2}\rangle|^p\right)^{\frac{1}{p}}\\
	&\quad \leq \|A\|_{ber_{\nabla_t}}\|k_{\lambda_1}\| \|k_{\mu_1}\|+\|B\|_{ber_{\nabla_t}}\|k_{\lambda_2}\|\|k_{\mu_2}\|\\
	&\quad \leq \max\{\|A\|_{ber_{\nabla_t}},\|B\|_{ber_{\nabla_t}}\}(\|k_{\lambda_1}\|\|k_{\mu_1}\|+\|k_{\lambda_2}\|\|k_{\mu_2}\|)\\
	&\quad \leq \max\{\|A\|_{ber_{\nabla_t}},\|B\|_{ber_{\nabla_t}}\}.
\end{split}
\end{equation*}
	This gives inequality $(i)$.\\ 
    $(ii)$ It follows from the fact that $\left\| P^*\begin{pmatrix}
    0 & A \\
	B & 0
	\end{pmatrix}P\right\|_{ber_{\sigma_t}}=\left\|\begin{pmatrix}
	0 & A \\
	B & 0
	\end{pmatrix}\right\|_{ber_{\sigma_t}},$ where $P=\begin{pmatrix}
	0 & I \\
	I & 0
	\end{pmatrix} \in {B}{(\mathcal{H} \oplus \mathcal{H})}$.\\
	$(iii)$	Let $T=	\begin{pmatrix}
	0 & A \\
	0 & 0
	\end{pmatrix}$. For any $(\lambda_1,\lambda_2),(\mu_1,\mu_2) \in \Omega  \times  \Omega ,$ let $\hat{k}_{(\lambda_1,\lambda_2)}=(k_{\lambda_1},k_{\lambda_2}),\\ \hat{k}_{(\mu_1,\mu_2)}=(k_{\mu_1},k_{\mu_2}) $ be two normalised reproducing kernels in $\mathcal{H}\oplus \mathcal{H}.$ Then we have 
\begin{equation*}
\begin{split}
	&|\langle T\hat{k}_{(\lambda_1,\lambda_2)}, \hat{k}_{(\mu_1,\mu_2)} \rangle|^p{\nabla_t}
	|\langle T^*\hat{k}_{(\lambda_1,\lambda_2)}, \hat{k}_{(\mu_1,\mu_2)} \rangle|^p\\
	&\qquad =t|\langle A{k}_{\lambda_2}, {k}_{\mu_1} \rangle|^p+(1-t)|\langle A^*{k}_{\lambda_1}, {k}_{\mu_2} \rangle|^p\\
	&\qquad \leq \|A\|^p_{ber}(t\|k_{\lambda_2}\|\|k_{\mu_1}\|+(1-t)\|k_{\lambda_1}\|\|k_{\mu_2}\|)\\
	&\qquad \leq \max \{t,1-t\}\|A\|^p_{ber}(\|k_{\lambda_2}\|\|k_{\mu_1}\|+\|k_{\lambda_1}\|\|k_{\mu_2}\|)\\
	&\qquad \leq \max \{t,1-t\}\|A\|^p_{ber}.
\end{split}
\end{equation*}
\end{proof}
	Using this proposition, we establish an upper bound for $2\times 2$ off-diagonal operator matrices.
\begin{corollary}
    Let $A,B \in {B\mathcal{(H)}}$ with $\sigma_t\leq\nabla_t$, then
    $$\left\|\begin{pmatrix}
    0 & A \\
   	B & 0
    \end{pmatrix}\right\|_{ber_{\sigma_t}} \leq \max\{t,1-t\}^{\frac{1}{p}} (\|A\|_{ber}+\|B\|_{ber}).$$
\end{corollary}
\begin{proof}
    We have 
\begin{equation*}
\begin{split}
    \left\|\begin{pmatrix}       
    0 & A \\
 	B & 0
    \end{pmatrix}\right\|_{ber_{\sigma_t}} &\leq \left\|\begin{pmatrix}
  	0 & A \\
    B & 0
	\end{pmatrix}\right\|_{ber_{\nabla_t}}\leq\left\|\begin{pmatrix}
   	0 & A \\
    0 & 0
    \end{pmatrix}\right\|_{ber_{\nabla_t}}+ \left\|\begin{pmatrix}          		 0 & 0 \\
    B & 0
    \end{pmatrix}\right\|_{ber_{\nabla_t}}\\
    &\leq \max\{t,1-t\}^{\frac{1}{p}}(\|A\|_{ber}+\|B\|_{ber}).
\end{split}
\end{equation*}
\end{proof}
     Next theorem also provides an upper bound for $2\times 2$ off-diagonal operator matrices. 
\begin{theorem}
    Let $A,B \in B\mathcal{(H)}.$ Let $f $and $g$ be two nonnegative continuous functions defined on $[0,\infty)$ such that $f(t)g(t)=t$ for every $t \geq 0$. If $\sigma_{t}\leq \nabla_t, $ then for $p\geq 1$
\begin{equation*}
\begin{split}
 	\left\|\begin{pmatrix}
   	0 & A \\
 	B & 0
    \end{pmatrix}\right\|^p_{ber_{\sigma_t}} &\leq \frac{2^p}{4} \max \bigg\{\textit{ber}\left(tf^{2p}(|A|)+(1-t)f^{2p}(|A^*|)\right)+\textit{ber}\left(tg^{2p}(|A^*|)+(1-t)g^{2p}(|A|)\right),\\
    &\qquad \qquad \quad \textit{ber}\left(tf^{2p}(|B|)+(1-t)f^{2p}(|B^*|)\right)+\textit{ber}\left(tg^{2p}(|B^*|)+(1-t)g^{2p}(|B|)\right)\bigg\}.
\end{split}
\end{equation*}
\end{theorem}
\begin{proof}
    Let $T=	\begin{pmatrix}
    0 & A \\
   	B & 0		
    \end{pmatrix}$.
    For any $(\lambda_1,\lambda_2),(\mu_1,\mu_2) \in \Omega  \times  \Omega ,$ let $\hat{k}_{(\lambda_1,\lambda_2)}=(k_{\lambda_1},k_{\lambda_2}),\\ \hat{k}_{(\mu_1,\mu_2)}=(k_{\mu_1},k_{\mu_2}) $ be two normalised reproducing kernels in $\mathcal{H}\oplus \mathcal{H}.$ Then we have
\begin{equation*}
\begin{split}
 	|\langle T\hat{k}_{(\lambda_1,\lambda_2)},\hat{k}_{(\mu_1,\mu_2)}\rangle |^p&=|\langle  Ak_{\lambda_2},k_{\mu_1}\rangle + \langle Bk_{\lambda_1},k_{\mu_2} \rangle|^p\\
    &\leq (|\langle Ak_{\lambda_2},k_{\mu_1} \rangle|+|\langle  Bk_{\lambda_1},k_{\mu_2} \rangle|)^p\\
   	&\leq \frac{2^p}{2}(|\langle Ak_{\lambda_2},k_{\mu_1} \rangle|^p+|\langle  Bk_{\lambda_1},k_{\mu_2} \rangle|^p)\\
   	&\leq \frac{2^p}{2}\! \big(\langle f^{2p}(|A|)k_{\lambda_2},k_{\lambda_2} \rangle^{\frac{1}{2}}\langle g^{2p}(|A^*|)k_{\mu_1},k_{\mu_1}\rangle^{\frac{1}{2}}\!\\
   	&\qquad +\!\langle f^{2p}(|B|)k_{\lambda_1},k_{\lambda_1} \rangle^{\frac{1}{2}}\langle g^{2p}(|B^*|)k_{\mu_2},k_{\mu_2}\rangle^{\frac{1}{2}}\big).
\end{split}	
\end{equation*}
    In a similar way, we get
\begin{equation*}
\begin{split}
	 |\langle T^*\hat{k}_{(\lambda_1,\lambda_2)},\hat{k}_{(\mu_1,\mu_2)}\rangle |^p &\!\leq\! \frac{2^p}{2}\! \bigg(\langle f^{2p}(|A^*|)k_{\lambda_2},k_{\lambda_2} \rangle^{\frac{1}{2}}\langle g^{2p}(|A|)k_{\mu_1},k_{\mu_1}\rangle^{\frac{1}{2}}\!\\
	 &\qquad +\!\langle f^{2p}(|B^*|)k_{\lambda_1},k_{\lambda_1} \rangle^{\frac{1}{2}}\langle g^{2p}(|B|)k_{\mu_2},k_{\mu_2}\rangle^{\frac{1}{2}}\bigg).
\end{split}
\end{equation*}
    Therefore 
\begin{equation*}
\begin{split}
    &|\langle  T\hat{k}_{(\lambda_1,\lambda_2)},\hat{k}_{(\mu_1,\mu_2)}\rangle |^p ~  \sigma_t ~ |\langle T^*\hat{k}_{(\lambda_1,\lambda_2)},\hat{k}_{(\mu_1,\mu_2)}\rangle |^p\\
   	&\qquad \leq t\left|\left\langle T\hat{k}_{(\lambda_1,\lambda_2)},\hat{k}_{(\mu_1,\mu_2)}\right\rangle \right|^p +(1-t)\left|\left\langle T^*\hat{k}_{(\lambda_1,\lambda_2)},\hat{k}_{(\mu_1,\mu_2)}\right\rangle \right|^p \\
   	&\qquad \leq \frac{2^p}{2} \bigg(t\Big(\left\langle f^{2p}(|A|)k_{\lambda_2},k_{\lambda_2} \right\rangle^{\frac{1}{2}}\left\langle g^{2p}(|A^*|)k_{\mu_1},k_{\mu_1}\right\rangle^{\frac{1}{2}}\\
   	&\qquad \quad +\left\langle f^{2p}(|B|)k_{\lambda_1},k_{\lambda_1} \right\rangle^{\frac{1}{2}}\left\langle g^{2p}(|B^*|)k_{\mu_2},k_{\mu_2}\right\rangle^{\frac{1}{2}}\Big)\\
    &\qquad\quad+(1-t)\Big(\left\langle f^{2p}(|A^*|)k_{\lambda_2},k_{\lambda_2} \right\rangle^{\frac{1}{2}}\left\langle g^{2p}(|A|)k_{\mu_1},k_{\mu_1}\right\rangle^{\frac{1}{2}}\\
    &\qquad \quad +\left\langle f^{2p}(|B^*|)k_{\lambda_1},k_{\lambda_1} \right\rangle^{\frac{1}{2}}\left\langle g^{2p}(|B|)k_{\mu_2},k_{\mu_2}\right\rangle^{\frac{1}{2}}\Big)\bigg)\\
    &\qquad = \frac{2^p}{2} \bigg (t\Big(\left\langle f^{2p}(|A|)k_{\lambda_2},k_{\lambda_2} \right\rangle^{\frac{1}{2}}\left\langle g^{2p}(|A^*|)k_{\mu_1},k_{\mu_1}\right\rangle^{\frac{1}{2}}\Big)\\
    &\quad \qquad+(1-t)\left(\left\langle f^{2p}(|A^*|)k_{\lambda_2},k_{\lambda_2}\right\rangle^{\frac{1}{2}}
    \left\langle g^{2p}(|A|)k_{\mu_1},k_{\mu_1}\right\rangle^{\frac{1}{2}}\right)\\
    &\quad \qquad  +t\Big(\left\langle f^{2p}(|B|)k_{\lambda_1},k_{\lambda_1} \right\rangle^{\frac{1}{2}}\left\langle g^{2p}(|B^*|)k_{\mu_2},k_{\mu_2}\right\rangle^{\frac{1}{2}}\Big)\\
    &\quad\qquad+(1-t)\Big(\left\langle f^{2p}(|B^*|)k_{\lambda_1},k_{\lambda_1} \right\rangle^{\frac{1}{2}}\left\langle g^{2p}(|B|)k_{\mu_2},k_{\mu_2}\right\rangle^{\frac{1}{2}}\Big)\bigg)\\
    &\qquad \leq \frac{2^p}{2} \bigg( \big\langle \left(tf^{2p}(|A|)+(1-t)f^{2p}(|A^*|)\right)k_{\lambda_2},k_{\lambda_2} \big \rangle^{\frac{1}{2}}\\
    &\quad\qquad
    \left\langle(tg^{2p}(|A^*|)+(1-t)g^{2p}(|A|))k_{\mu_1},k_{\mu_1} \right\rangle^{\frac{1}{2}}\\
    &\quad\qquad  +\left\langle (tf^{2p}(|B|)+(1-t)f^{2p}(|B^*|))k_{\lambda_1},k_{\lambda_1} \right\rangle^{\frac{1}{2}}\\
    &\quad\qquad\left\langle (tg^{2p}(|B^*|)+(1-t)g^{2p}(|B|))k_{\mu_2},k_{\mu_2} \right\rangle^{\frac{1}{2}}\bigg)\\
    &\qquad \leq\frac{2^p}{2} \bigg( \Big( \textit{ber}^{\frac{1}{2}}\!\!\left(tf^{2p}(|A|) +(1-t)f^{2p}(|A^*|)\right)\\
    &\quad\qquad\qquad\quad \textit{ber}^{\frac{1}{2}}\!\!\left(tg^{2p}(|A^*|)+(1-t)g^{2p}(|A|)\right)\Big)\|k_{\lambda_2}\|\|k_{\mu_1}\|\\
    &\quad\qquad + \Big(\textit{ber}^{\frac{1}{2}}\!\!\left(tf^{2p}(|B|)+(1-t)f^{2p}(|B^*|)\right)\\
    &\quad\qquad\qquad\quad\textit{ber}^{\frac{1}{2}}\!\!\left(tg^{2p}(|B^*|)+(1-t)g^{2p}(|B|)\right)\Big)\|k_{\lambda_1}\|\|k_{\mu_2}\|\bigg)\\
    &\qquad\leq \frac{2^p}{2} \max \bigg\{\textit{ber}^{\frac{1}{2}}\!\!\left(tf^{2p}(|A|)+(1-t)f^{2p}(|A^*|)\right)\textit{ber}^{\frac{1}{2}}\!\!\left(tg^{2p}(|A^*|)+(1-t)g^{2p}(|A|)\right),\\
    &\qquad\qquad \qquad \quad \textit{ber}^{\frac{1}{2}}\!\!\left(tf^{2p}(|B|)+(1-t)f^{2p}(|B^*|)\right)\textit{ber}^{\frac{1}{2}}\!\!\left(tg^{2p}(|B^*|)+(1-t)g^{2p}(|B|)\right)\bigg\}\\
    &\qquad\qquad\qquad\quad\times \left(\|k_{\lambda_2}\|\|k_{\mu_1}\|+\|k_{\lambda_1}\|\|k_{\mu_2}\|\right)\\
    &\qquad\leq \frac{2^p}{4} \max \bigg\{\textit{ber}\left(tf^{2p}(|A|)+(1-t)f^{2p}(|A^*|)\right)+\textit{ber}\left(tg^{2p}(|A^*|)+(1-t)g^{2p}(|A|)\right),\\
    &\qquad\qquad \qquad \quad \textit{ber}\left(tf^{2p}(|B|)+(1-t)f^{2p}(|B^*|)\right)+\textit{ber}\left(tg^{2p}(|B^*|)+(1-t)g^{2p}(|B|)\right)\bigg\}.
\end{split}	
\end{equation*} 
    Taking the supremum, we get the desired inequality.
\end{proof}
\begin{corollary}
    Let $A,B \in B\mathcal{(H)}$. If $\sigma_{t}\leq \nabla_t, $ then for $p\geq 1$
\begin{equation*}
\begin{split}
    \left\|\begin{pmatrix}
   	0 & A \\
   	B & 0
    \end{pmatrix}\right\|^p_{ber_{\sigma_t}} &\leq \frac{2^p}{4} \max \bigg\{\textit{ber}\left(|A|^p+|A^*|^p\right), \textit{ber}\left(|B|^p+|B^*|^p\right)\bigg\}.
\end{split}
\end{equation*}
\end{corollary}
	Now we obtain an upper bound for an arbitrary $2 \times 2$ operator matrices.
\begin{theorem}
	Let $A,B,C,D \in B\mathcal{(H)}.$ Then 
	$$
	\left\|\begin{pmatrix}
	A & B \\
	C & D
	\end{pmatrix}\right\|_{ber_{\sigma_t}} \leq \left\|\begin{pmatrix}
	\left\|A\right\|_{ber_{\nabla_t}} & \left(t\left\| B\right\|^p_{ber}+(1-t)\left\|C\right\|^p_{ber}\right)^{\frac{1}{p}} \\
	\left(t\left\| C\right\|^p_{ber}+(1-t)\left\|B\right\|^p_{ber}\right)^{\frac{1}{p}} & \left\|D\right\|_{ber_{\nabla_t}}
	\end{pmatrix}\right\|.
	$$
\end{theorem}
\begin{proof}
	Let $T=	\begin{pmatrix}
	A & B \\
	C & D		
	\end{pmatrix}$.
	For any $(\lambda_1,\lambda_2),(\mu_1,\mu_2) \in \Omega  \times  \Omega ,$ let $\hat{k}_{(\lambda_1,\lambda_2)}=(k_{\lambda_1},k_{\lambda_2}),\\ \hat{k}_{(\mu_1,\mu_2)}=(k_{\mu_1},k_{\mu_2}) $ be two normalised reproducing kernels in $\mathcal{H}\oplus \mathcal{H}.$ Then we have 
\begin{equation*}
\begin{split}
	&\left(\left|\left\langle T\hat{k}_{(\lambda_1,\lambda_2)}, \hat{k}_{(\mu_1,\mu_2)} \right\rangle\right|^p 
	~\sigma_t~ 
	\left|\left\langle T^*\hat{k}_{(\lambda_1,\lambda_2)}, \hat{k}_{(\mu_1,\mu_2)} \right\rangle\right|^p
	\right)^{\frac{1}{p}}\\
	&\quad = \bigg(\left|\left\langle A k_{\lambda_1},k_{\mu_1}\right\rangle+\left\langle B k_{\lambda_2},k_{\mu_1}\right\rangle+\left\langle C k_{\lambda_1},k_{\mu_2}\right\rangle+\left\langle D k_{\lambda_2},k_{\mu_2}\right\rangle\right|^p \\ 
	&\qquad  \sigma_t~ \left|\left\langle A^* k_{\lambda_1},k_{\mu_1}\right\rangle+\left\langle C^* k_{\lambda_2},k_{\mu_1}\right\rangle+\left\langle B^* k_{\lambda_1},k_{\mu_2}\right\rangle+\left\langle D^* k_{\lambda_2},k_{\mu_2}\right\rangle\right|^p \bigg)^{\frac{1}{p}}\\
	&\quad\leq \bigg(t\big(\left|\left\langle A k_{\lambda_1},k_{\mu_1}\right\rangle+\left\langle B k_{\lambda_2},k_{\mu_1}\right\rangle+\left\langle C k_{\lambda_1},k_{\mu_2}\right\rangle+\left\langle D k_{\lambda_2},k_{\mu_2}\right\rangle\right|^p \big)\\ 
	&\qquad +(1-t) \big(\left|\left\langle A^* k_{\lambda_1},k_{\mu_1}\right\rangle+\left\langle C^* k_{\lambda_2},k_{\mu_1}\right\rangle+\left\langle B^* k_{\lambda_1},k_{\mu_2}\right\rangle+\left\langle D^* k_{\lambda_2},k_{\mu_2}\right\rangle\right|^p\big) \bigg)^{\frac{1}{p}}\\
	&\quad \leq \left(t\left|\left\langle A k_{\lambda_1},k_{\mu_1}\right\rangle\right|^p+(1-t)\left|\left\langle A^* k_{\lambda_1},k_{\mu_1}\right\rangle\right|^p\right)^{\frac{1}{p}}+\left(t\left|\left\langle B k_{\lambda_2},k_{\mu_1}\right\rangle\right|^p+(1-t)\left|\left\langle C^* k_{\lambda_2},k_{\mu_1}\right\rangle\right|^p\right)^{\frac{1}{p}}\\
	&\qquad +\left(t\left|\left\langle C k_{\lambda_1},k_{\mu_2}\right\rangle\right|^p+(1-t)\left|\left\langle B^* k_{\lambda_1},k_{\mu_2}\right\rangle\right|^p\right)^{\frac{1}{p}}+\left(t\left|\left\langle D k_{\lambda_2},k_{\mu_2}\right\rangle\right|^p+(1-t)\left|\left\langle D^* k_{\lambda_2},k_{\mu_2}\right\rangle\right|^p \right)^{\frac{1}{p}}\\
	&\quad \leq \left\|A\right\|_{ber_{\nabla_t}}\left\|k_{\lambda_1}\right\|\left\|k_{\mu_1}\right\|+\left\|D\right\|_{ber_{\nabla_t}}\left\|k_{\lambda_2}\right\|\left\|k_{\mu_2}\right\|\\
	&\qquad +\left(t\left\| B\right\|^p_{ber}+(1-t)\left\|C\right\|^p_{ber}\right)^{\frac{1}{p}} \left\|k_{\lambda_2}\right\|\left\|k_{\mu_1}\right\|+\left(t\left\| C\right\|^p_{ber}+(1-t)\left\|B\right\|^p_{ber}\right)^{\frac{1}{p}} \left\|k_{\lambda_1}\right\|\left\|k_{\mu_2}\right\|\\
	&\quad = \left\langle \begin{pmatrix}
	\left\|A\right\|_{ber_{\nabla_t}} & \left(t\left\| B\right\|^p_{ber}+(1-t)\left\|C\right\|^p_{ber}\right)^{\frac{1}{p}} \\
	\left(t\left\| C\right\|^p_{ber}+(1-t)\left\|B\right\|^p_{ber}\right)^{\frac{1}{p}} & \left\|D\right\|_{ber_{\nabla_t}}
	\end{pmatrix}\begin{pmatrix}
	\left\|k_{\lambda_1}\right\| \\
	\left\|k_{\lambda_2}\right\|
	\end{pmatrix},\begin{pmatrix}
	\left\|k_{\mu_1}\right\| \\
	\left\|k_{\mu_2}\right\|
	\end{pmatrix}\right\rangle \\
	&\quad \leq \left\|\begin{pmatrix}
	\left\|A\right\|_{ber_{\nabla_t}} & \left(t\left\| B\right\|^p_{ber}+(1-t)\left\|C\right\|^p_{ber}\right)^{\frac{1}{p}} \\
	\left(t\left\| C\right\|^p_{ber}+(1-t)\left\|B\right\|^p_{ber}\right)^{\frac{1}{p}} & \left\|D\right\|_{ber_{\nabla_t}}
	\end{pmatrix}\right\|.
\end{split}
\end{equation*}
   Therefore, taking the supremum over all $\lambda ,\mu \in \Omega,$ we obtain the desired result.
\end{proof}  
   The next result gives a characterization of $2 \times 2$ diagonal invertible operator matrices that are unitary.
\begin{theorem}
    Suppose
    $
    T=
    \begin{pmatrix}
   	A & 0 \\
    0 & B
    \end{pmatrix}
    \in {B}(\mathcal{H}\oplus\mathcal{H})
    $
    is invertible. Then $T$ is unitary if and only if
    $
    \max\left\{
    \textit{ber}(A^*A),\,
    \textit{ber}(B^*B)
    \right\}\leq 1
    \quad \text{and} \quad
    \max\left\{
    \textit{ber}\big((A^*A)^{-1}\big),\,
    \textit{ber}\big((B^*B)^{-1}\big)
    \right\}\leq 1.
    $
\end{theorem}
\begin{proof}
	Suppose that $T$ is unitary. Then $A$ and $B$ are unitary, which implies
    $
	\textit{ber}(A^*A)\!=1,$
	$\textit{ber}\big((A^*A)^{-1}\big)=1,
	$
	and
	$
    \textit{ber}(B^*B)=1,
	\textit{ber}\big((B^*B)^{-1}\big)=1.
	$
	Conversely, suppose that
	$
	\max\left\{
	\textit{ber}(A^*A),\,
	\textit{ber}(B^*B)
	\right\}\leq 1
	$
	and
	$
	\max\left\{
    \textit{ber}\big((A^*A)^{-1}\big),
	\textit{ber}\big((B^*B)^{-1}\big)
	\right\}\leq 1.
	$
	We compute
	$
	\big\|(T-T^{-1^*})\hat{k}_{(\lambda_1,\lambda_2)}\big\|^2.
	$
	Using the inequalities
	$
	\textit{ber}(T^*T)
	\leq
	\max\left\{
	\textit{ber}(A^*A),
	\textit{ber}(B^*B)
	\right\}\\
	\leq 1
	$
	and
	$
	\textit{ber}\big((T^*T)^{-1}\big)
	\leq
	\max\left\{
	\textit{ber}\big((A^*A)^{-1}\big),\,
	\textit{ber}\big((B^*B)^{-1}\big)
	\right\}
	\leq 1,
	$
	we obtain
	$
	\big\|(T-T^{-1^*})\hat{k}_{(\lambda_1,\lambda_2)}\big\|=0 \text{ for all } (\lambda_1,\lambda_2)\in\Omega\times\Omega.
	$
	This gives $T^{-1}=T^*$.
\end{proof}
\section{convexity of the berezin range}
    In this section, we study the convexity of the Berezin range of composition operators and finite rank operators on weighted Hardy space and Fock space.
     
     We compute the Berezin range of the composition operators and finite rank operators on weighted Hardy space corresponding to different weight sequences and investigate conditions under which their Berezin ranges are convex. Furthermore, we examine the convexity of the Berezin range of the composition operators on Fock space over $\mathbb{C}^n$ with  symbol $\phi(z)=Az,$ where $A$ is a scalar matrix of order $n$ and $z=(z_1,z_2,...,z_n)$.
     
    \subsection{On weighted Hardy space}
     
    Let $(\beta_n)_{n \ge 0}$ be a sequence such that $\beta_0 = 1$, 
    $\beta_n > 0$ for all $n$, and 
    $
    \liminf_{n \to \infty} \beta_n^{1/n} \ge 1.
    $
    Then the weighted Hardy space $H^2(\beta)$ is defined by
    $$
    H^2(\beta) 
    = \left\{
    f(z) = \sum_{n=0}^{\infty} a_n z^n 
    \; : \;
    \sum_{n=0}^{\infty} |a_n|^2 \beta_n^2 < \infty
    \right\}.
    $$
    The inner product on $H^2(\beta)$ is given by
    $$
    \langle f, g \rangle 
    = \sum_{n=0}^{\infty} a_n \overline{b_n}\,\beta_n^2,
    $$
    where $f(z) = \sum_{n=0}^{\infty} a_n z^n$ and 
    $g(z) = \sum_{n=0}^{\infty} b_n z^n$.
     
    It is well known that $H^2(\beta)$ is a reproducing kernel 
    Hilbert space consisting of analytic functions on $\mathbb{D}$, 
    with reproducing kernel
    $$
    k_w(z) 
    = \sum_{n=0}^{\infty} \frac{\overline{w}^{\,n}}{\beta_n^2}\, z^n.
    $$
       
    By choosing different weight sequences, we get different function spaces. For example, the classical Hardy space, the classical Bergman space and the classical Dirichlet space are weighted Hardy spaces with weights $\beta_n=1, ~\beta_n=(n+1)^{-\frac{1}{2}}$ and $\beta_n=(n+1)^{\frac{1}{2}}$ respectively.  
    See \cite[Chapter 2]{cowen1995composition} for more details.
     
    Let  $\phi : \mathbb{D}\rightarrow \mathbb{D}$ be a holomorphic function. A composition operator $C_\phi$ acting on ${H}^{2}(\beta)$ is defined by $ C_\phi f := f \circ \phi$ is bounded. The Berezin transform of $C_{\phi}$ on $w$ is given by
    $$
    \widetilde{C_{\phi}}(w) = \frac{1}{\|k_w\|^2} \langle C_{\phi} k_w, k_w \rangle
    = \frac{1}{\|k_w\|^2} k_w(\phi(w))
    =\frac{\displaystyle \sum_{n=0}^{\infty} \frac{\overline{w}^n\phi(w)^n}{\beta_n^2}}{\displaystyle \sum_{n=0}^{\infty} \frac{|w|^{2n}}{\beta_n^2}}.
    $$
    
    If we consider $\beta_n = \left( \binom{n + k - 1}{n} \right)\!^{-\frac{1}{2}}$ for a fixed $k \geq 1$, then
    $$
    k_w(z)  = \sum_{n=0}^\infty  \binom{n + k - 1}{n} \overline{w}^n z^n = \frac{1}{(1-\overline{w}z)^k}.
    $$
    When $k=1,$ this is  the classical Hardy space and when  $k>1,$ this is the weighted Bergman space.  The convexity of the Berezin range of operators on spaces associated with this kernel has already been studied in the literature. In particular, the convexity of the Berezin range of finite rank operators on this space is discussed in \cite{finiterank}, while the convexity of the Berezin range of composition operators is investigated in \cite{anirben}.
     
    If we consider $\beta_n^2 = (\frac{1}{\beta})^n$ for $\beta \in (0,1]$, then
    $$
    k_w(z)=\sum_{n=0}^{\infty}\beta^n\overline{w}^n z^n=\frac{1}{1-\beta\overline{w}z}.
    $$
    From here onward, all our discussions take place in the weighted Hardy space with  weight $\beta_n^2 = (\frac{1}{\beta})^n$.
       
    The Berezin transform of the composition operator associated with this kernel is given by
    $$
    \widetilde{C_{\phi}}(w) =  \frac{1-|w|^2\beta}{1-\overline{w}\phi(w) \beta}.
    $$
    Let $\phi(w) = \eta w$ be a holomorphic self map on $\mathbb{D}$, 
    where $\eta \in \mathbb{\overline{D}}$ and $w \in \mathbb{D}$. Then
    $$
    \widetilde{C_{\phi}}(w) = 
    \frac{1-|w|^2\beta}{1-|w|^2 \eta \beta}.
    $$
    The Berezin range of these operators are not always convex (see Figure \ref{figure3}). Here we characterise the convexity of  $\textit{Ber}(C_\phi)$. 

\begin{figure}[H]
	\centering
	\includegraphics[scale=0.5]{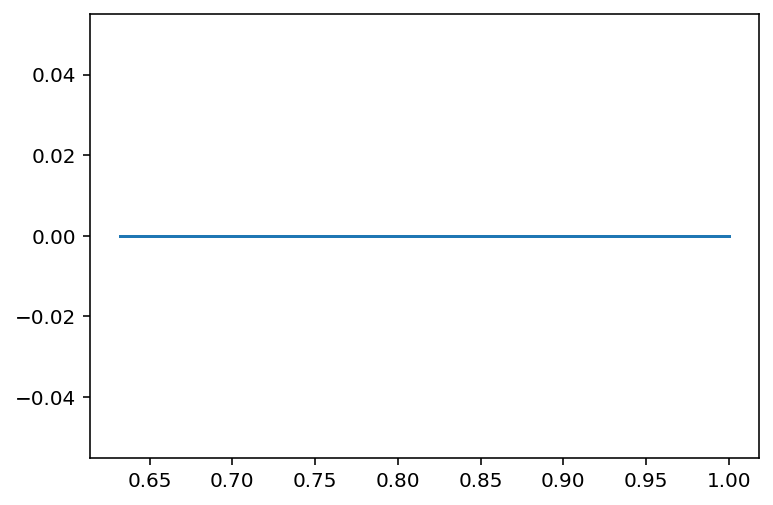}
	\includegraphics[scale=0.5]{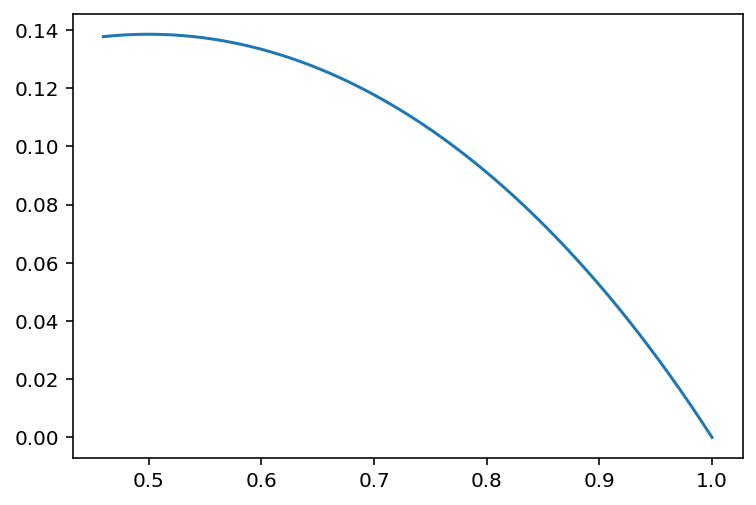}
	\caption{$\textit{Ber}(C_{\phi})$ on $H^2(\beta)$ for $\eta=-0.75$ , $\beta=0.25$ (left, apparently  convex) and $\eta=0.6i$ , $\beta=0.5$ (right, apparently not convex)}
	\label{figure3}
\end{figure}
\begin{theorem}\label{4.1}
    Let $ C_\phi \in {B}(H^2(\beta)) $ be such that $\phi(w) = \eta w$ with $\eta \in \mathbb{\overline{D}}$ and $w \in \mathbb{D}$.   Then $\textit{Ber}(C_\phi)$ is convex if and only if  $\eta \in [-1,1]$.
\end{theorem}
\begin{proof}
    Suppose that \(\eta=1\), then \( \phi(w)=w\). Putting \( w=re^{i\theta}\) for \(0 \leq r < 1\), we get
    \[
    \widetilde{C_{\phi}}(w) = \frac{1-r^2\beta}{1-r^2 
    \beta}=1.
    \]
    So \(\textit{Ber}(C_\phi)\)=\{1\}, which is convex. Similarly, for \(\phi(w)=\eta w\), where $-1 \leq \eta < 1,$ we obtain
    \[
    \textit{Ber}(C_\phi) = \left\{ \frac{1 - r^2 \beta}{1 - r^2 \eta  \beta} : r \in \left[0, 1\right) \right\} = \left( \frac{1 - \beta}{1 -\eta  \beta}, 1 \right],
    \]
    which is also convex.
     	
    Conversely, suppose that \(\textit{Ber}(C_\phi)\) is convex. We have to show that $\eta \in [-1,1]$. We have 
    \[
    \widetilde{C_{\phi}}(re^{i\theta})  = \frac{1-|r|^2\beta}{1-|r|^2 
     		\eta \beta}     \qquad  r \in [0,1),
    \]
    which is independent of \(\theta\). Therefore \(\textit {Ber}(C_\phi)\) is a path in \(\mathbb{C}\). The convexity implies that it can be either a point or a line segment. It is easy to observe that  \(\textit {Ber}(C_\phi)\) is a point if and only if \(\eta=1\). Now, consider \(\textit {Ber}(C_\phi)\) as a line segment. Note that \(\widetilde{C_{\phi}}(0)=1\) and \(\lim_{r \to 1^{-}} \widetilde{C_{\phi}}(re^{i\theta}) = \frac{1-\beta}{1-\beta\eta}\). Since \(\beta \not=0\) these two are distinct points in \(\textit {Ber}(C_\phi)\). The convexity of \(\textit {Ber}(C_\phi)\) implies that the points 1, \(\frac{1-\beta}{1-\beta\eta}\) and an arbitrary point \(\frac{1-\beta\rho^2}{1-\beta\rho^2\eta}\) for \(\rho \in [0,1)\) are collinear. In the \(\mathbb{R}^2\) plane these points can be seen as  \((1,0),\left(\frac{(1-\beta)(1-\beta a)}{(1-\beta a)^2+(\beta b)^2},\frac{(1-\beta)\beta b}{(1-\beta a)^2+(\beta b)^2}\right)\) and \( \left( \frac{(1-\beta\rho^2)(1-\beta\rho^2a)}{(1-\beta\rho^2a)^2+(\beta\rho^2b)^2},\frac{(1-\beta\rho^2)\beta\rho^2b}{(1-\beta\rho^2a)^2+(\beta\rho^2)^2}\right)\) respectively, where \(\eta=a+ib\). From collinearity, we have
    $$
    1  \left( \frac{(1 - \beta) \beta b}{(1 - \beta a)^2 + (\beta b)^2} - \frac{(1 - \beta \rho^2) \beta \rho^2 b}{(1 - \beta \rho^2 a)^2 + (\beta \rho^2 b)^2} \right) 
    + \frac{(1 - \beta)(1 - \beta a)}{(1 - \beta a)^2 + (\beta b)^2}  \left( \frac{(1 - \beta \rho^2) \beta \rho^2 b}{(1 - \beta \rho^2 a)^2 + (\beta \rho^2 b)^2} \right) $$
    $$
    + \frac{(1 - \beta \rho^2)(1 - \beta \rho^2 a)}{(1 - \beta \rho^2 a)^2 + (\beta \rho^2 b)^2}  \left( - \frac{(1 - \beta) \beta b}{(1 - \beta a)^2 + (\beta b)^2} \right) = 0,
    $$
    which implies
    $$
    \frac{ \big((1-\beta)\beta b\big)\big((1-\beta \rho^2 a)^2+(\beta \rho^2 b)^2\big) - \big((1-\beta \rho^2)\beta \rho^2 b\big)\big((1-\beta a)^2+(\beta b)^2\big)  }{ \big((1-\beta a)^2+(\beta b)^2\big)\big((1-\beta \rho^2 a)^2+(\beta \rho^2 b)^2\big) } $$
    $$
    + \frac{(1-\beta)(1-\beta a )(1-\beta \rho^2)(\beta \rho^2 b)}{ \big((1-\beta a)^2+(\beta b)^2\big)\big((1-\beta \rho^2 a)^2+(\beta \rho^2 b)^2\big) } 
    + \frac{(1-\beta \rho^2)(1-\beta \rho^2 a)(\beta - 1)(\beta b)}{ \big((1-\beta a)^2+(\beta b)^2\big)\big((1-\beta \rho^2 a)^2+(\beta \rho^2 b)^2\big) } = 0.
    $$
    Simplifying, we get
    $$
    \frac{2\big(\beta^3\rho^2ab-\beta^3\rho^4ab-\beta^3\rho^2b+\beta^3\rho^4b\big)}{\big((1-\beta a)^2+(\beta b)^2\big)\big((1-\beta \rho^2a)^2+(\beta \rho^2 b)^2\big)}=0,
    $$
    implies
    $$
    b\big(a(\beta^3\rho^2-\beta^3\rho^4)-(\beta^3\rho^2-\beta^3\rho^4)\big)=0.
    $$
    This implies that either \(b=0\) or \(a=1\). In both cases the imaginary part of \(\eta\) is 0. Therefore \(\eta \in[-1,1]\). This completes the proof.
\end{proof}
     
    For elliptic automorphism \(\phi(w) = \eta w\) with \(\eta \in \mathbb{T}\), we have the following result:
\begin{corollary}
    Let \( C_\phi \in {B}(H^2(\beta)) \) be such that \(\phi(w) = \eta w\) with \(\eta \in \mathbb{T}\) and \(w \in \mathbb{D}\).   Then \(\textit{Ber}(C_\phi)\)  is convex if and only if  \(\eta=1\) or \(\eta=-1\).
\end{corollary}
\begin{remark}
    By setting $\beta=1$ in Theorem~\ref{4.1}, we get \cite[Theorem~4.1]{augustine2023composition}, which states that if $\eta \in \overline{\mathbb{D}}$ and $\phi(w)=\eta w$, then the Berezin range of $C_{\phi}$ on $H^2(\mathbb{D})$ is convex if and only if $-1 \leq \eta \leq 1$.
\end{remark}
    \textbf{Blaschke Factor:} Consider the automorphism on the unit disc,  known as the Blaschke factor:
    \[
    \phi_\alpha(z) = \frac{z - \alpha}{1 - \overline{\alpha} z},
    \]
    where \(\alpha\) \(\in\) \(\mathbb{D}\).
    We have
    $$
    \widetilde{C}_{\phi_\alpha}(w) = \frac{1 - |w|^2\beta}{1 - \overline{w} \phi_\alpha(w)\beta}
    =\frac{(1 - |w|^2 \beta)(1 - \overline\alpha {w})}{1 - \overline\alpha {w} - |w|^2 \beta+  \alpha \overline{w} \beta}.
    $$
\begin{proposition}
    For the composition operator $C_{\phi_\alpha}$ on ${H}^2(\beta)$, the Berezin range $\textit{Ber}(C_{\phi_\alpha})$ is closed under complex conjugation and therefore is symmetric about the real axis.
\end{proposition}
\begin{proof}
    Let $w = r e^{i\varphi}$ and $\alpha = \rho e^{i\theta}$. We show that
    \[
    \widetilde{C}_{\phi_{\alpha}}(r e^{i\varphi}) = \overline{\widetilde{C}_{\phi_{\alpha}}(r e^{i(2\theta - \varphi)})},
    \]
    i.e.,
    \[
    \frac{1 - r^2\beta}{1 - r e^{-i\varphi} \phi_{\alpha}(r e^{i\varphi})\beta} 
    =
    \overline{ \frac{1 - r^2\beta}{1 - r e^{-i(2\theta - \varphi)} \phi_{\alpha}(r e^{i(2\theta - \varphi)})\beta}} .
    \]
    This occurs only if
    \[
    e^{-i\varphi} \phi_{\alpha}(r e^{i\varphi}) = \overline{e^{-i(2\theta - \varphi)} \phi_{\alpha}(r e^{i(2\theta - \varphi)})}.
    \]
    Now,
\begin{equation*}
	\begin{split}
    e^{2i\theta} \overline{\phi_{\alpha}(r e^{i(2\theta - \varphi)})} 
     &= e^{2i\theta} \overline{\left( \frac{r e^{i(2\theta - \varphi)} - \beta}{1 - \bar{\beta} r e^{i(2\theta - \varphi)}} \right)} \\
    &= e^{2i\theta} \left( \frac{r e^{i(\varphi - 2\theta)} - \rho e^{-i\theta}}{1 - \rho e^{i\theta} r e^{i(\varphi - 2\theta)}} \right) \\
    &= \frac{r e^{i\varphi} - \rho e^{i\theta}}{1 - \rho e^{-i\theta} r e^{i\varphi}} \\
    &= \phi_{\alpha}(r e^{i\varphi}).
\end{split}     	
\end{equation*}
    This completes the proof.
\end{proof}
\begin{corollary}
    If the Berezin range of \(C_{\phi_\alpha}\) is convex, then  \(\Re\{\widetilde{C}_{\phi_\alpha}(z) \} \in \textit{Ber}({C_{\phi_\alpha}})\) for each z \(\in \mathbb{D}\).
\end{corollary}
\begin{proof}
    Suppose \(\textit{Ber}({C_{\phi_\alpha}})\) is convex. Then by the above proposition, we have \(\textit{Ber}({C_{\phi_\alpha}})\) is closed under complex conjugation. Therefore, we have 
    \[
    \frac{1}{2} C_{{\phi_\alpha}}(z) + \frac{1}{2} \overline{C_{{\phi_\alpha}}(z)} 
    = \Re\left\{ C_{{\phi_\alpha}}(z) \right\} \in \operatorname{Ber}(C_{\phi_\alpha}).
    \]
\end{proof}
    The Berezin range of this operator need not  always be convex (see Figure \ref{figure2}). 
\begin{figure}[H]
    \centering
    \includegraphics[scale=0.7]{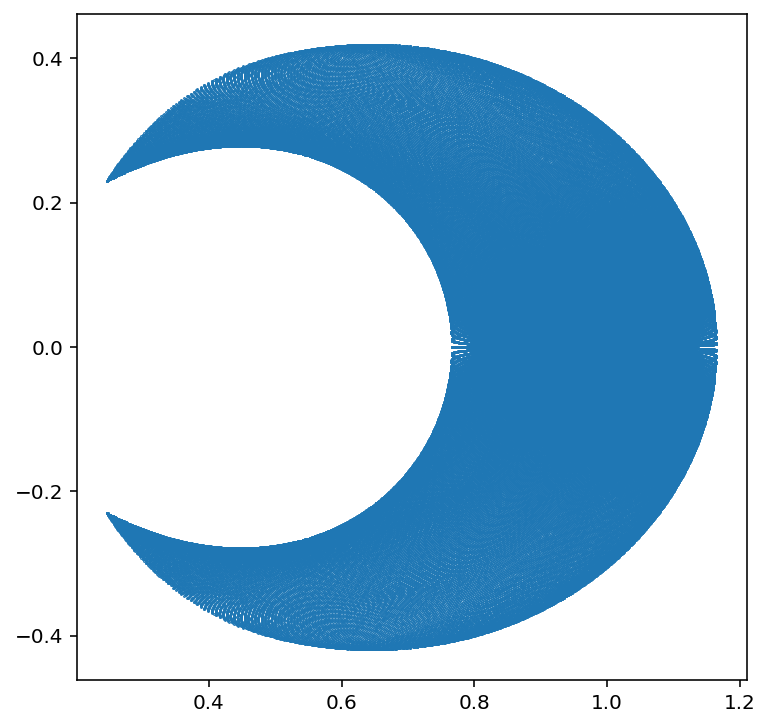}
    \caption{The Berezin range $\textit{Ber}(C_{\phi_{\alpha}})$ on $H^2(\beta)$ for $\alpha=0.5$ and $\beta=0.7$ ( apparently not  convex).}
    \label{figure2}
\end{figure}
\begin{remark}
    When $\alpha=0$,     $$\widetilde{C}_{\phi_{\alpha}}(w)=\frac{(1 - |w|^2 \beta)}{1  - |w|^2 \beta  }=1\qquad \forall ~ w \in \mathbb{D},$$ which implies $\textit{Ber}(C_{\phi_{\alpha}})=\{1\},$  a convex subset of $\mathbb{C}$. However, a complete characterization of 
    the values of $\alpha$ for which the Berezin range $\textit{Ber}(C_{\phi_{\alpha}})$ is convex 
     is still open.
     \begin{problem}
     	For what values of $\alpha$ the Berezin range $\textit{Ber}(C_{\phi_{\alpha}})$ is convex? 
     \end{problem}
\end{remark}
    Now we discuss  the convexity of the Berezin range of finite rank operators. 
    First, we prove the convexity of the Berezin range of a rank one operator of the form \(A(f)=\langle f, z^n \rangle z^n\), for \(n \in \mathbb{N}\).
\begin{theorem}\label{4.8}
    Let \(A(f)=\langle f, z^n \rangle z^n\) be a rank one operator on \(H^2(\beta)\). Then the Berezin range of A is \(\textit{Ber}(A)=\big[0,(\frac{n}{\beta(n+1)})^n (\frac{1}{n+1})\big]\), which is convex in \(\mathbb{C}\).
\end{theorem}
\begin{proof}
    For \(\lambda \in \mathbb{D}\), \(A(k_\lambda)=\langle k_\lambda,z^n \rangle z^n=\overline{\lambda}^n z^n\). The Berezin transform at \(\lambda\) is
    \[
    \widetilde{A}(\lambda)=(1-\beta|\lambda|^2)\langle \overline{\lambda}^n z^n,k_\lambda \rangle =(1-\beta|\lambda|^2) |\lambda|^{2n} =|\lambda|^{2n}-\beta |\lambda|^{2n+2},
    \]
    where \(|\lambda| \in [0,1)\). Observe that \(A(\lambda)=A(|\lambda|)\) is a real function. Now we differentiate \(A(\lambda)\) with respect to \(|\lambda|\) and equate it to zero to find the extreme points.
    \[
    0=2n|\lambda|^{2n-1}-\beta(2n+2)|\lambda|^{2n+1}=2|\lambda|^{2n-1}(n-\beta(n+1)|\lambda|^2).
    \]
    This happens if and only if \(|\lambda|=0  \text{ or } |\lambda|^2=\frac{n}{\beta (n+1)}\). Now if \(|\lambda|=0\), then \(\widetilde{A}(\lambda)=0\). If \(|\lambda|^2=\frac{n}{\beta(n+1)}\), then
    \[
    \widetilde{A}(\lambda)=\left(\frac{n}{\beta(n+1)}\right)^n-\beta \left(\frac{n}{\beta(n+1)}\right)^{n+1}=\left(\frac{n}{\beta(n+1)}\right)^n \left(\frac{1}{n+1}\right).
    \]
    Thus, the Berezin range of A is \(\textit{Ber}(A)=\big[0,(\frac{n}{\beta(n+1)})^n (\frac{1}{n+1})\big]\), which is a convex subset of \(\mathbb{C}\).
\end{proof}
\begin{remark}
     By setting $\beta=1$ in Theorem~\ref{4.8}, we get \cite[Proposition 2.1]{athulfiniterank}, which states that if $ A(f) = \langle f, z^n \rangle z^n
     $ is a rank one operator on  $H^2(\mathbb{D}).$
     Then the Berezin range of $A$ is 
     $
     \textit{Ber}(A)
     =
     \left[
     0,\,
     \left(\frac{1}{n+1}\right)
     \left(\frac{n}{n+1}\right)^{\!n}
     \right],
     $
     which is  convex in $\mathbb{C}$.
\end{remark}
     Now we prove the general case of the above theorem by choosing any arbitrary $ g_i \in H^2(\beta).$
\begin{theorem}
     Let A be the finite rank operator
     \(
     A(f)=\sum_{i=1}^n\bigl\langle f,\,g_i\bigr\rangle\,g_i\),
     where \(f,g_i\in H^2(\beta)\). Then the Berezin range of \(A\),
     \(\textit
     {Ber}(A)\)
     is a convex subset of \(\mathbb{C}\).
\end{theorem}
\begin{proof}
    For \(\lambda\in\mathbb{D}\), the action of \(A\) on the reproducing kernel \(k_\lambda\) is
    \[
    A(k_\lambda)
    =\sum_{i=1}^n\langle k_\lambda,\,g_i\rangle\,g_i(z)
    =\sum_{i=1}^n \overline{g_i(\lambda)}\,g_i(z).
    \]
    Hence, the Berezin transform at \(\lambda\) is
    \[
    \widetilde A(\lambda)
    =(1-\beta |\lambda|^2)\bigl\langle A(k_\lambda),\,k_\lambda\bigr\rangle
    =(1-\beta |\lambda|^2)\sum_{i=1}^n|g_i(\lambda)|^2,
    \quad \lambda\in\mathbb{D}.
    \]
    Since each \(g_i\in H^2(\beta)\) is holomorphic, \(\widetilde A(\lambda)\) is a real-valued continuous function on the connected set \(\mathbb{D}\). Therefore, its image
    \(\textit{Ber}(A)\) is connected in \(\mathbb{R}\). In \(\mathbb{R}\), the only connected subsets are intervals or singletons. Hence, \(\textit{Ber}(A)\) must be either an interval or a singleton. Therefore,  \(\textit{Ber}(A)\) is a convex subset of \(\mathbb{C}\).
\end{proof}
    Next we check the convexity of the Berezin range of the operator of the form \(A(f)=\langle f, z^n \rangle z^m, \text{ where } m>n\).
\begin{theorem}\label{4.11}
    Let \(A(f)=\langle f, z^n \rangle z^m, \text{ where } m>n\) be a rank one operator on \(H^2(\beta)\). Then  $\textit{Ber}(A)$ is a disc with centre at the origin and radius $\left(\frac{2}{m+n+2}\right)\left(\frac{m+n}{\beta(m+n+2)}\right)^{\frac{m+n}{2}}$ and therefore is convex in $\mathbb{C}.$.
\end{theorem}
\begin{proof}
    For \(\lambda \in \mathbb{D}\), \(A(k_\lambda)=\langle k_\lambda,z^n \rangle z^m=\overline{\lambda}^n z^m\). The Berezin transform at \(\lambda\) is
    \[
    \widetilde{A}(\lambda)=(1-\beta|\lambda|^2)\langle \overline{\lambda}^n z^m,k_\lambda \rangle =(1-\beta|\lambda|^2) |\lambda|^{2n} \lambda^{m-n}.
    \]
    Put \(\lambda = re^{i\theta}\). Then 
    we get
    \[\widetilde{A}(\lambda)=(1-\beta r^2) r^{2n} r^{m-n} e^{i(m-n)\theta}=(r^{m+n}-\beta r^{m+n+2})e^{i(m-n)\theta}.
    \]
    Therefore, the Berezin range of A is 
    \[
    \textit{Ber}(A)=\left\{ (r^{m+n}-\beta r^{m+n+2})e^{i(m-n)\theta}:re^{i\theta} \in \mathbb{D}\right\}.
    \]
     	
    For each \(r\in[0,1)\), this set is a circular set. For \(\eta\in \textit{Ber}(A)\), we have
    \(
    \eta=(r^{m+n}-\beta r^{m+n+2})e^{i\theta(m-n)}  \) for some 
    \(re^{i\theta}\in \mathbb{D}.
    \)
     Then for any \(t\in[0,2\pi)\),
    \(\eta e^{it}
    =(r^{m+n}-\beta r^{m+n+2})e^{i\theta(m-n)}e^{it}
    =(r^{m+n}-\beta r^{m+n+2})e^{i(\theta(m-n)+t)}\in\textit{Ber}(A),
    \) since \(\eta e^{it}\) is the image of \(re^{i(\theta(m-n)+t)}\in \mathbb{D}\).  Since \(0\in\textit{Ber}(A)\),
    it is easy to observe that \(\textit{Ber}(A)\) is the disc centered at the origin with radius \(\sup_{r\in[0,1)}\bigl(r^{m+n}-\beta r^{m+n+2}\bigr)\).
    To find the extreme values of \((r^{m+n}-\beta r^{m+n+2})\), we differentiate it and equate the derivative to 0. We get 
    \[
    0=(m+n)r^{m+n-1}-\beta (m+n+2)r^{m+n+1}=r^{m+n-1}\left((m+n)-\beta (m+n+2)r^2\right).
    \]
    Therefore, the extreme points are \(r=0 \text{ and } r= \sqrt{\frac{m+n}{\beta(m+n+2)}}\). Substituting these values, we get the radius of the disc as \(\left(\frac{2}{m+n+2}\right)\left(\frac{m+n}{\beta(m+n+2)}\right)^{\frac{m+n}{2}}\). Hence 
    \[ 
    \textit{Ber}(A)=\mathbb{D}_{\left(\frac{2}{m+n+2}\right)\left(\frac{m+n}{\beta(m+n+2)}\right)^{\frac{m+n}{2}}},
    \]
    which is convex in \(\mathbb{C}\).
\end{proof}
\begin{remark}
    By setting $\beta=1$ in Theorem~\ref{4.11},  we get \cite[Theorem 2.10]{athulfiniterank}, which states that if $A(f)=\langle f,z^n\rangle z^m, m >n $ is a rank one operator on $H^2(\mathbb{D})$.
    Then the Berezin range of $A$, $\textit{Ber}(A)$ is a disc with centre at the origin and radius $\left(\frac{2}{m+n+2}\right)\left(\frac{m+n}{(m+n+2)}\right)^{\frac{m+n}{2}}$ and therefore is convex in $\mathbb{C}.$
\end{remark}

\subsection{ On Fock Space}
     
    For any \( \alpha > 0 \), consider the Gaussian probability measure
    \[
    dv_\alpha(z) =  \frac{\alpha}{{\pi}^n}  e^{-\alpha |z|^2} \, dv(z)
    \]
    on \( \mathbb{C}^n \), where \( dv \) is the Lebesgue volume measure on \( \mathbb{C}^n \). The Fock space  \( \mathcal{F}^2_\alpha(\mathbb{C}^n) \) \cite{carswell2003composition} consists of all holomorphic functions \( f \) on \( \mathbb{C}^n\) with
    \[
    \|f\|^2_\alpha \equiv \int_{\mathbb{C}^n} |f(z)|^2 \, dv_\alpha(z) < \infty.
    \]
    \( \mathcal{F}^2_\alpha(\mathbb{C}^n) \) is a reproducing kernel  Hilbert space with the following inner product
    \[
    \langle f, g \rangle_\alpha = \int_{\mathbb{C}^n} f(z) \overline{g(z)} \, dv_\alpha(z).
    \]
    The reproducing kernel at $w\in \mathbb{C}^n $ is given by
    \[
    k_w(z) = k(z, w) = e^{\alpha \langle z, w \rangle},
    \]
    where $\langle z,w \rangle=\sum_{i=1}^{n}z_i\overline{w}_i$ and 
    \[
    \|k_w\|^2_\alpha = e^{\alpha \|w\|^2}.
    \] 
     
    In \cite{carswell2003composition}, Carswell, MacCluer and Schuster characterised the bounded composition operators on the Fock space $\mathcal{F}^2_\alpha(\mathbb{C}^n).$
\begin{theorem}\cite{carswell2003composition}
    Suppose $\phi:\mathbb{C}^n\to \mathbb{C}^n$ is  a holomorphic mapping.
\begin{itemize}
    \item[(a)] If $C_{\phi} $ is bounded on $\mathcal{F}^2_\alpha(\mathbb{C}^n)$, then $\phi(z)=Az+B$,  where $A$ is an $n \times n$ matrix and $B$ is an $n \times 1$ vector. Furthermore, $\|A\|\leq 1, $ and if $|A\zeta|=|\zeta|$ for some $\zeta \in \mathbb{C}^n,$ then $\langle A\zeta , B\rangle=0.$
    \item[(b)] If $C_{\phi} $ is compact on $\mathcal{F}^2_\alpha(\mathbb{C}^n)$, then $\phi(z)=Az+B$, where $\|A\|<1.$ 
\end{itemize}
    Converse is also true.
\end{theorem}
     
    If we consider \(\phi(z)=Az\), where \(A\) is a  matrix of order \(n\) and \(z=(z_1,z_2,...,z_n)\), then the Berezin transform of the composition operator is given by 
\begin{align*}
    \widetilde{C}_\phi(w)= \langle C_\phi \hat k_w ,\hat k_w\rangle
    =\frac{1}{\|k_w\|^2}\langle C_\phi k_w,k_w\rangle
    =\frac{1}{e^{\alpha|w|^2}}k_w(\phi(w)),
\end{align*}
    where $w=(w_1,w_2,...,w_n)$. When $A=\lambda I$, a scalar matrix of order $n$, we get
    $$
    \widetilde{C}_\phi(w)=e^{\sum_{i=1}^{n} (\lambda - 1)\,\alpha\, |w_i|^2}.
    $$
     
    The Berezin range of these operators is not always convex, as we see in Figure \ref{figure1}. Here we try  to find the values of $\alpha$ for which the Berezin range is convex.
     
\begin{figure}[H]
    \centering
    \includegraphics[scale=0.7]{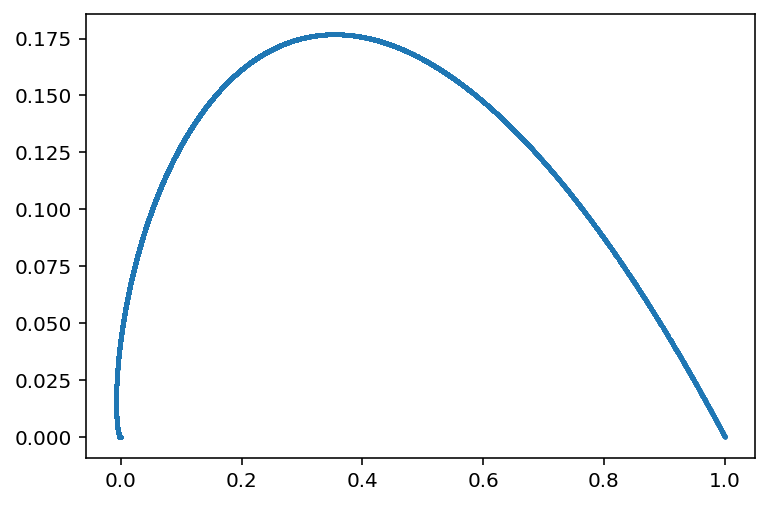}
    \caption{The Berezin range $\textit{Ber}(C_{\phi})$ on $\mathcal{F}^2_{\alpha}(\mathbb{C}^n)$ for $\alpha=1$ and $\lambda=0.5i$ ( apparently not convex).}
    \label{figure1}
\end{figure} 
    The following theorem characterizes the convexity of the Berezin range of this operator.
\begin{theorem}
    Let \(z \in \mathbb{C}^n\) and \( \lambda \in \overline{\mathbb{D}}\) and \( \phi(z)=A z\), where \(A=\lambda I\). Then the Berezin range of \( C_\phi\) acting on \(\mathcal{F}^2_\alpha(\mathbb{C}^n)\) is convex if and only if \( \lambda \in [-1,1].\)
\end{theorem}
\begin{proof}
    Suppose that \( \lambda =1\). Then \(\phi(z)=Iz\), which implies that
    \[
    \widetilde{C}{_\phi(z)}=e^0=1.
    \]
    Thus, \(\textit{Ber}(C_\phi)= \{ 1 \}\), which is a convex set in \( \mathbb{C}\).
    Now suppose that \(\lambda \in [-1,1)\). Let \( z_i=r_ie^{i\theta_i}\) for \(i=1,2,...,n\). Then for \( 0\leq r_i < \infty\)
    \[
    \textit{Ber}(C_\phi)=\big\{ e^{\sum_{i=1}^{n} (\lambda - 1)\,\alpha\, r_i^2}:r_i \in [0, \infty)\big\}=(0,1],
    \]
    which is also convex in \(\mathbb{C}\).
     	
    Conversely, suppose that \(\textit{Ber}(C_\phi)\) is convex. We need to prove that \(\lambda \in [-1,1].\) Let \(\lambda-1=a+ib  \). Then
    \[
    \widetilde{C}_\phi(r_1e^{i\theta_1},r_2e^{i\theta_2},...,r_ne^{i\theta_n})=e^{\sum_{i=1}^{n}(a+ib)\alpha r_i^2}=e^{\sum_{i=1}^{n}a\alpha r_i^2}e^{\sum_{i=1}^{n}ib\alpha r_i^2}, \qquad r\in [0,\infty).
    \]
    If \(a=0\), then \(\lambda = 1+ib\) and since \(\lambda \in \overline{\mathbb{D}}\) we have \(b=0\).
     	
    In order to prove $ \lambda \in [-1,1]$, it is enough to prove that if the imaginary part of \( \lambda\) is non zero then  \(\textit{Ber}(C_\phi)\) is not convex. For  non zero $a$ and $b$ we have
    \[
    \widetilde{C}_\phi(r_1e^{i\theta_1},r_2e^{i\theta_2},...,r_ne^{i\theta_n})=e^{\sum_{i=1}^{n}a\alpha r_i^2}e^{\sum_{i=1}^{n}ib\alpha r_i^2},
    \]
    which is a function independent of \(\theta\). Therefore \(\textit{Ber}(C_\phi)\) is just a path in \(\mathbb{C}\). Our assumption \(\textit{Ber}(C_\phi)\) is convex  implies  that  it is either a point or a line segment.  It is easy to observe that \(\widetilde{C}_\phi(0,0,...,0) = 1 \), 
    \( \widetilde{C}_\phi(\frac{1}{\sqrt{\alpha}},0,...,0) = e^{ a} e^{i b} \) and \( \widetilde{C}_\phi(\rho e^{i\theta},0,...,0) = e^{\alpha a \rho ^2} e^{i\alpha b \rho ^2} \) where \(\rho \in [0,\infty)\).
    Since \( a \) and \( b \) are non-zero real numbers, 1 and \( e^a e^{ib} \) are two distinct points in the complex plane. Thus, \( \textit{Ber}(C_\phi) \) is a line segment. Then, every points in \( \textit{Ber}(C_\phi) \) are collinear.
    Therefore  the points 1 , \( e^{ a} e^{ ib } \) and \( e^{\alpha a \rho ^2} e^{i\alpha b \rho ^2}\)   in \( \textit{Ber}(C_\phi) \) are collinear. In the \( \mathbb{R}^2 \) plane, these three points can be viewed as \( (1, 0),\\ (e^{a} \cos( b), e^{a} \sin ( b)) \) and \( (e^{a\alpha \rho^2} \cos(b\alpha\rho^2), e^{a\alpha\rho^2} \sin(b\alpha\rho^2)) \) respectively. 
    Then, from collinearity, we have
\begin{equation*}
\begin{split}
    &1 \Big(e^{ a} \sin  b -  e^{\alpha a\rho^2} \sin(\alpha b\rho^2)\Big) 
    + e^{ a} \cos b \Big(e^{\alpha a\rho^2} \sin(\alpha b\rho^2) - 0\Big)\\
    &\qquad \qquad + e^{\alpha a\rho^2} \cos(\alpha b\rho^2) \Big(0 - e^{ a }\sin { b}\Big) = 0.
\end{split}
\end{equation*}
    Rearranging and simplifying the above equation, we get  
    \[
    e^{ a} \sin b = e^{\alpha a\rho^2} \big( \sin(\alpha b\rho^2) + e^{a} \sin(\alpha b\rho^2 - b)\big) \qquad \forall~\rho \in [0, \infty) .
    \]  
    In particular, put \( \rho = \sqrt{\frac{2\pi }{\alpha b} } \) in the above equation. Then  
    $$
    e^{ a} \sin  b = e^{\frac{a\pi}{b}} [\sin 2\pi + e^{ a} \sin(2\pi + b)]
    = e^{\frac{a\pi}{b}} [  e^{ a} \sin b].
    $$
    By cancelling \( e^{ a} \sin  b \) on both sides, we get  
    \[
    e^{\frac{a\pi}{b}} = 1=e^0.
    \]
    This implies that \( a \pi = 0 \), thus \(a=0\), which contradicts our assumption that \( a \neq 0 \).  
    So \( \textit{Ber}(C_\phi) \) is not convex.  
    Therefore, if the imaginary part of \( \lambda \) is non-zero, then  
    \( \textit{Ber}(C_\phi) \) is not convex.
\end{proof}
    Next example shows that the Berezin range of a composition operator with symbol $\phi(z)=Az$, where A is a non-constant diagonal matrix, need not  always be convex.
\begin{example}
   	Consider the Fock space $\mathcal{F}^2(\mathbb{C}^2)$. Suppose that $\phi(z)=Az$, where $
   	A=
   	\begin{pmatrix}
    1 & 0 \\
   	0 & i
   	\end{pmatrix}$. Then the Berezin transform of  the composition operator is given by
   	$$
   	\widetilde{C}{_\phi(z_1,z_2)}=e^{(i-1)|z_2|^2}.
   	$$ 
   	Thus
   	$$
   	\textit{Ber}(C_{\phi})=\{e^{(i-1)|z_2|^2}:z_2 \in \mathbb{C}\}=\{e^{-t}e^{it}:t\geq 0\}.
   	$$
   	When $t=0,t=\pi$ the points $1,-e^{-\pi} \in  \textit{Ber}(C_{\phi}) .$
   	We will show that the midpoint
   	$$
   	m=\frac{1-e^{-\pi}}{2} \notin  \textit{Ber}(C_{\phi}).
   	$$
   	Suppose that $m \in  \textit{Ber}(C_{\phi})$. Then there exists some $t\geq0$ such that
   	$m=e^{-t}e^{it}$. Since $m>0,$ we have $t=2k\pi $ for some $k \in \mathbb{N} \cup \{0\}. $ Then $m=e^{-2k\pi }$ for some $k$. That is $m\in \{1,e^{-2\pi},e^{-4\pi},...\}$, but $\frac{1-e^{-\pi}}{2} \notin \{1,e^{-2\pi},e^{-4\pi},...\}.
   	$
\end{example}
\begin{theorem}
    For a fixed $k\in\{1,2,...,n\},$ consider the $n\times n$ diagonal matrix $A_k=[a_{kj}]$ with entries $a_{kk}=a+ib$ and $a_{jj}=1$ $\forall$ $k\not=j$, where $a^2+b^2\leq 1$. Then the Berezin range of the composition operator with symbol $\phi(z)=A_kz$ is convex if and only if $b=0$.   
\end{theorem}
\begin{proof}
    Suppose that $b=0$. Then the Berezin
    range is given by
    $$
    \textit{Ber}(C_{\phi})=\{e^{\alpha(a-1)|z_i|^2}:z_i \in \mathbb{C}\}=(0,1],
    $$
    which is a convex subset of $\mathbb{C}.$
    
    Conversely suppose that $\textit{Ber}(C_{\phi})$ is convex. We have to show that b=0. Suppose $b\not=0$. It is enough to prove that $\textit{Ber}(C_{\phi})$ is not convex. We have 
    $$\widetilde{C}_{\phi}(z_1,z_2,...,z_n)=e^{\alpha(a-1)|z_i|^2}e^{ib|z_i|^2}=e^{\alpha(a-1)t}e^{i\alpha bt},\quad t\geq 0.$$
    When $t=0 \text{ and } t= \frac{2\pi}{b\alpha}, $ we get $1,e^{(a-1)\frac{2\pi}{b} }\in \textit{Ber}(C_{\phi})$. Thus the midpoint $m=\frac{1+e^{(a-1)\frac{2\pi}{b} }}{2}\in \textit{Ber}(C_{\phi}).$ Then $m=e^{\alpha(a-1)t}e^{i\alpha bt}$ for some $t\geq0.$  Since $m>0, ~\alpha bt=2k\pi$ for some $k\in \mathbb{N}\cup\{0\}$. This implies $m=e^{(a-1)}\frac{2k\pi}{b}$, i.e, $m \in \{1,e^{(a-1)}\frac{2\pi}{b},e^{(a-1)}\frac{4\pi}{b},...\}.$ But $\frac{1+e^{(a-1)\frac{2\pi}{b} }}{2}\not\in \{1,e^{(a-1)}\frac{2\pi}{b},e^{(a-1)}\frac{4\pi}{b},...\}$. Hence, $\textit{Ber}(C_{\phi})$ is not convex.
\end{proof}
    \textbf{Declaration of competing interest}
	
	There is no competing interest.\\
	
	\textbf{Data availability}
	
	No data was used for the research described in the article.\\
	
	{\bf Acknowledgments.}   The first author is supported by the Junior Research Fellowship of UGC (University Grants Commission, India). The second author is supported by the Senior Research Fellowship (09/0239(13298)/2022-EMR-I) of CSIR (Council of Scientific and Industrial Research, India).
	 
	\nocite{*}
	\bibliographystyle{amsplain}
	
\end{document}